\documentclass[10pt]{article}
\usepackage{amsmath}
\usepackage{amssymb}
\usepackage{amsfonts}
\usepackage{amscd}
\usepackage[usenames]{color}
\usepackage{verbatim}
\usepackage{amsthm}
\usepackage{graphicx,psfrag}
\usepackage{epstopdf}
\usepackage{tikz}
\usetikzlibrary{chains}

\input{xypic}
\xyoption{all}

\renewcommand{\theequation}{\arabic{section}.\arabic{equation}}
\newtheorem{theorem}{Theorem}[section]

\newtheorem{lemma}[theorem]{Lemma}
\newtheorem{cor}[theorem]{Corollary}

\newtheorem{remark}[theorem]{Remark}

\newtheorem{example}[theorem]{Example}

\newtheorem{VanishingTheorem}[theorem]{Vanishing Theorem}

\def\ep{\varepsilon}
\def\la{\lambda}
\def\w{\omega}

\def\R{{\mathbb R}}
\def\cx{{\mathbb C}}
\def\P{{\mathbb P}}
\def\V{{\mathcal V}}

\renewcommand{\cal}{\mathcal}
\def\C{{\cal C}}
\def\CP{{\cal P}}
\def\D{{\cal D}}
\def\E{{\cal E}}
\def\F{{\cal F}}

\def\L{{\cal L}}
\def\M{{\cal M}}
\def\N{{\cal N}}
\def\O{{\cal O}}
\def\X{{\cal X}}
\def\CZ{{\cal Z}}

\def\del{\overline \partial}

\def\dim{\operatorname{dim }}
\def\ker{\operatorname{ker\,}}
\def\cok{\operatorname{coker\,}}
\def\Aut{\operatorname{Aut}}
\def\sgn{\operatorname{sgn}}

\def\ind{\operatorname{index\;}}

\def\Re{\operatorname{Re}}

\def\non{\noindent}

\def\ssetminus{\hspace{-.02cm}\setminus\hspace{-.03cm}}

\def\wh#1{\widehat{#1}}
\def\Rf{{\cal R}_f}

\def\tcup{{\textstyle \bigcup}}

\def\f{{\bf f}}
\def\p{{\bf p}}
\def\r{{\bf r}}

\def\sgn{{\rm sgn}}
\def\dimc{\operatorname{dim}_\cx}

\tikzset{ch/.style={node distance=2.3em,circle,draw,on chain,inner sep=1.2pt},chj/.style={ch,join},line width=2pt,baseline=-1ex}
\newcommand{\LabeledCircle}[3][chj]{
\node[#1,label={above: #2},label={below: #3}] {};
}

\newcommand{\DiagramDots}{
\node[chj,draw=none,inner sep=1pt,node distance=1em] {\dots};
}

\newcommand*\samethanks[1][\value{footnote}]{\footnotemark[#1]}

\date{\empty}
\addtocounter{section}{1}
\begin{document}

\title{\bf Recursion Formulas for Spin Hurwitz Numbers}
\vskip.2in

\author{ Junho Lee\thanks{ partially supported by the N.S.F.} 
\and
Thomas H. Parker\samethanks}


\maketitle

\begin{abstract}
The classical Hurwitz numbers   which count coverings of a complex curve have an analog when the curve is endowed with a theta characteristic.  These
``spin Hurwitz numbers'', recently studied by Eskin, Okounkov and Pandharipande, are interesting in their own right.  By the authors' previous work, they are also related to the Gromov-Witten invariants of K\"{a}hler surfaces.
 We prove a recursive formula for spin Hurwitz numbers, which then gives the dimension zero GW invariants of K\"{a}hler surfaces  with positive geometric genus.   The proof uses a degeneration of spin curves, an invariant defined by the spectral flow of  certain anti-linear deformations of  $\del$,   and an interesting localization phenomenon for eigenfunctions that  shows that maps with even ramification points cancel in pairs.
\end{abstract}


\vspace{6mm}


\label{Introduction}
\bigskip

The  Hurwitz numbers  of a complex curve $D$ count covers with specified ramification type.   Specifically, consider
 degree $d$  (possibly disconnected) covering maps  $f:C\to D$ with fixed ramification  points $q^1, \dots, q^k\in D$ and ramification given by  $m^1, \dots, m^k$ where each  $m^i=(m^i_1,\cdots,m^i_{\ell_i})$ is a partition of $d$. The   Euler characteristic  of $C$ is  related to  the genus $h$ of $D$ and   the partition lengths $\ell(m^i)=\ell_i$  by the Riemann-Hurwitz formula
\begin{equation}
\label{introdim0}
\chi(C)\ =\ 2d(1-h)+ \sum_{i=1}^k \big(\ell(m^i)-d).
\end{equation}
In this context, there is an ordinary Hurwitz number
\begin{equation}
\label{introplainHurwitz}
\sum \frac{1}{|\Aut(f)|}
\end{equation}
that counts the covers $f$ satisfying (\ref{introdim0}) mod automorphisms; the sum depends only on $h$ and   $\{m^i\}$.

Now fix a theta characteristic $N$ on  $D$, that is, a holomorphic line bundle with an isomorphism $N^2=K_D$ where $K_D$ is the canonical bundle of $D$.  The pair $(D,N)$ is called spin curve.   By a  well-known theorem of Mumford and Atiyah,   the deformation class of the spin curve is completely characterized by the genus $h$ of $D$ and the {\em parity}
\begin{equation}
\label{introparityD}
p\ =\ (-1)^{h^0(D, N)}
\end{equation}

Now consider degree $d$  ramified covers $f:C\to D$ for which
\begin{equation}
\label{introbullet}
\begin{minipage}{4.2in}
\begin{itemize}
\item each partition $m^i$ is odd, i.e. each $m^i_j$ is an odd number.
\end{itemize}
\end{minipage} \hspace{2cm}
\end{equation}
In this case, the ramification divisor $\Rf$ of $f$ is even and  the twisted pullback bundle
\begin{equation}\label{thetaonC}
N_f \ = \ f^*N\otimes \O\big(\tfrac12 \Rf\big)
\end{equation}
is a theta characteristic on $C$ with  parity
\begin{equation}
\label{introparity}
p(f)\ =\ (-1)^{h^0(C, N_f)}.
\end{equation}
After choosing a spin curve $(D,N)$ and odd partitions $m^1, \dots, m^k$, we can consider  the total count of maps  satisfying (\ref{introdim0}) modulo automorphisms, counting each map  as $\pm 1$ according to its parity.
This sum is  also a deformation invariant   of the spin curve $(D,N)$, so depends only on   $h$  and $p$.   Thus we define the {\em spin Hurwitz numbers} of a spin curve $(D,N)$ of genus $h$ and parity $p$ to be
\begin{equation}
\label{SHN}
H_{m^1,\cdots,m^k}^{h,p}\ =\
\sum \frac{p(f)}{|\Aut(f)|}\ \ \ \ \
\end{equation}
where the sum is over all non-isomorphic maps $f$ satisfying (\ref{introdim0}).

\smallskip

 Eskin, Okounkov and  Pandharipande \cite{EOP} gave a combinatorial method for finding the spin Hurwitz numbers when  $D$ is  an elliptic curve with the trivial theta characteristic  (genus $h=1$ and  parity $p= -1$).   Our main result gives recursive formulas that express all other spin Hurwitz numbers (except the related $h=0$ and $h=p=1$ cases) in terms of the Eskin-Okounkov-Pandharipande numbers.   The statement   involves two numbers that are associated with    partition $m=(m_1,\cdots,m_\ell)$ of  $d$, namely
$$
|m|\,=\,\prod m_j\ \ \ \ \
\mbox{and}\ \ \ \ \
m!\,=\,|\Aut(m)|
$$
where $\Aut(m)$.
We call a partition $m$  {\em odd} or {\em even} according to whether $|m|$ is odd or even.

\begin{theorem}\label{Main}
Fix $d>0$ and let $m^1,\cdots,m^k$ be a collection of odd partitions of $d$.
\begin{itemize}
\item[(a)] If $h=h_1+h_2$ and $p\equiv p_1+p_2$ (mod 2) then  for $0\leq k_0\leq k$
\begin{equation}\label{main1}
H_{m^1,\cdots,m^k}^{h,p}\ =\
\sum_m\,|m| m!\ H_{m^1,\cdots,m^{k_0},m}^{h_1,p_1}\cdot \, H_{m,m^{k_0+1},\cdots,m^k}^{h_2,p_2}
\end{equation}
\item[(b)] If $h\geq 2$ or if $(h,p)=(1,+)$ then
\begin{equation}\label{main2}
H_{m^1,\cdots,m^k}^{h,p}\ =\
\sum_m\,|m| m!\   H_{m,m,m^1,\cdots,m^k}^{h-1,p}
\end{equation}
\end{itemize}
where the sums are over all odd partitions $m$ of $d$.
\end{theorem}

Theorem~\ref{Main} applies, in particular, to the spin Hurwitz numbers  that count degree $d$ etale covers, defined as above by  taking $m$ to be the trivial partition $(1^d)$ of $d$.    These  etale  spin Hurwitz numbers $H_d^{h,p}=H^{h, p}_{(1^d)}$ are related to the GW invariants of complex projective surfaces, as follows.

 Let $X$ be such a surface with a smooth canonical divisor $D$.  By the adjunction formula,  the normal bundle $N\to D$ is a theta characteristic, so each component of  $(D, N)$  is  a spin curve.  The results of  $\cite{KL}$ and  \cite{LP1} show that  the  GW invariant of $X$ is a sum over the components of $(D, N)$ of certain {\em local  GW invariants} $GW^{loc}_{g,n}$.   As usual,  one can  work either with the local GW invariants   that count maps from connected domains of genus $g$ or with the local  `Gromov-Taubes' invariants  $GT^{loc}_{g,n}$ that count maps from possibly disconnected domains of Euler characteristic $\chi$.  With the latter, the main formula of \cite{LP1} reads
    \begin{equation}\label{introkey}
GT_{\chi,n}(X,\beta) \ =\ \prod_k (i_k)_*\,GT^{loc}_{\chi_k,n}(D_k, N_k; d_k)
\end{equation}
where $i_k$ is the inclusion $D_k\subset X$.

Now, assume $(D,N)$ is a connected  genus $h$  spin curve with parity $p$ and consider maps $f:C\to D$ where
$\chi(C)= 2d(h-1)$.
Then the space of degree $d$ stable maps  with no marked points has dimension zero,  both sides of (\ref{introkey}) are rational numbers and, in fact,
 the dimension zero local  GT local invariants are exactly the etale spin Hurwitz numbers:
\begin{equation}\label{introbydef}
GT_{d}^{loc,h,p}\,=\,H_d^{h,p}
\end{equation}
(the relation  $\chi= 2d(h-1)$ is implicit in this notation).  For $h=0,1$, these  invariants were calculated for all degrees $d$ in \cite{KL} and \cite{LP1}.
As an immediate corollary to Theorem~\ref{Main}, one can express the local invariants (\ref{introbydef}) with $h\geq 2$  in terms of
  $h=1$    spin Hurwitz numbers  calculated in \cite{EOP}:

\begin{cor}
\label{MainThm2}
Let $H_m$ denote the  spin Hurwitz numbers $H^{1,-}_m$ where $m$ is one or more partitions.   Then for $h\geq 2$ we have
$$
GT_d^{loc,h,p}\,=
\left\{
\begin{array}{ll}
{\displaystyle
\sum \, \prod_{i=1}^{h-1} \,|m^i| m^i!\, H_{m^{h-1}} \cdot H_{m^{h-1},m^{h-2}} \cdots\ H_{m^2,m^1} \cdot H_{m^1}
} & \mbox{if}\ \ h\equiv p \ (\mbox{mod}\ 2) \\
{\displaystyle
\sum \, \prod_{i=1}^{h-1} \,|m^i| m^i!\, H_{m^{h-1},m^{h-1},m^{h-2}}\cdot  H_{m^{h-2},m^{h-3}} \cdots\ H_{m^2,m^1} \cdot H_{m^1}
}
&  \mbox{if}\ \ h\not\equiv p \ (\mbox{mod}\ 2)
\end{array}
\right.
$$
where the sums are over all odd partitions $m^1,\cdots,m^{h-1}$ of $d$.
\end{cor}

 \medskip

 The proof of Theorem~\ref{Main} involves five main steps, described below.  All are based on the observation that the $\del$-operators on spin bundles $N_f$ extend to a 1-parameter family of {\em real} operators 
   \begin{equation}
\label{introLf}
L_t=\del+tR:\Omega^0(C, N_f)\to \Omega^{0,1}(C, N_f)
\end{equation}
with remarkable properties.  The key idea is that the parity of a map $f$ is an isotopy invariant of the family $L_t$,  and the one can explicitly describe the behavior of the operators $L_t$ 
as both the domain  and the target of $f$ degenerate to  nodal spin curves.  This allows us to express both the parity and the number of covering maps in terms of the maps into the irreducible components of the nodal target curve, giving the recursion formulas of Theorem~\ref{Main}.

\medskip 

\noindent{\sc Step 1:}  {\bf Relating $L_t$ and parity}.  Section~\ref{section2} gives method for constructing complex anti-linear bundle maps $R$, which then define a family $L_t=\del+tR$ of operators as in (\ref{introLf}).  We then prove a vanishing theorem showing that $\ker L_t=0$  for each stable map $f$ and each $t\not= 0$.   This property was exploited in our previous work (e.g. \cite{LP1}, \cite{LP2}) and underlies  all later sections.  

 In Section~\ref{section3}  we express the parity as an isotopy invariant --- the ``TR  spectral flow'' ---  of the path of operators $L_t$.  In this form, in contrast to the original definition (\ref{introparity}),  parity is unchanged under  deformations.  We also relate the parity to the determinant of $L_t$ on its low eigenspaces.

\medskip 

\noindent{\sc Step 2:} {\bf Degenerating spin curves and  sum formulas.}  The Hurwitz numbers of  $D$  can be viewed as  the relative  Gromov-Witten invariants  of $D$ relative to a   branch locus $\{q_1, \dots, q_\ell\}\in D$.  Under condition (\ref{introdim0})  the space of  relative stable maps  is a finite set  corresponding to  stable maps $f:C\to D$ branched over $\{q_i\}$.  We then adopt the sum formula   arguments of \cite{IP2}, as the first author has done  in  \cite{L}. There are three parts of the argument:
 \begin{itemize}
 \item Identifying the maps $f:C\to D$ that occur as limits  as $D$ degenerates to a nodal spin curve $D_0$.
\item  Constructing a family $\C\to\Delta$  of  deformations of the maps $f:C_0\to D_0$.
\item A gluing procedure that relates  the moduli space of a general fiber to data   along the central fiber.
\end{itemize}
 In each step, it is necessary to  keep track of the target curve, the domain curve, the map,   the spin structures,  and ultimately the spectral flow.  The spin structure adds complication: in order to extend the spin structure across the central fiber it is necessary, following Cornabla \cite{C}, to insert a rational curve at each node as the target degenerates. Section~\ref{section4}  proves Theorem~\ref{Main} assuming two deferred facts:  the existence of a smooth family moduli space and  a crucial statement (Theorem~\ref{MainTask}) about parities.  
 
\medskip 

\noindent{\sc Step 3:}  {\bf Algebraic families of maps.}   The required family of maps is constructed in Section~\ref{section5}.  The construction, which uses blowups and base changes,  provides explicit coordinates for the analysis done in later sections.    Extra steps are needed to ensure that there is are line bundles on the family whose restrictions to the general fiber gives the spin structure $N$ on $D$ and (\ref{thetaonC}) for each $f:C\to D$.   Moreover, as shown in Section~\ref{section6}, there are anti-linear bundle maps $R$, and hence operators $L_t=\del+tR$  on the family with the properties described in Section~\ref{section2}.

\medskip 
 
 \noindent{\sc Step 4}: {\bf  Eigenbundles of $L_t$ and parity for odd partitions.}    In Section~\ref{section7} we switch from algebraic geometry to analysis and  construct  bundles of  low eigenspaces of $L_t$.    The formulas of Section~\ref{section3} then apply on the family, giving a simple parity formula (Lemma~\ref{oddLemma})  for odd partitions.  But  a  complication arises for even partitions:  the maps into $D_0$ may be ramified over the nodes in a way that does not satisfy (\ref{introbullet}), so  the  irreducible components of $D_0$ do not have well-defined spin Hurwitz numbers.  Correspondingly, we obtain an analytic formula for the parity (Theorem~\ref{evenThm}) that must be evaluated at smooth curves.

 \medskip 
 
\noindent{\sc Step 5}: {\bf Localization and cancellation.}  Finally, we exploit another remarkable property of the operators $L_t$:  as  $t\to \infty$ there is a basis of the low eigenspace of   $L^*_t$ consisting of   ``bump functions'' sharply concentrated at the nodes, and $p(f)$ can be expressed in terms of $L^2$ inner products of these bump functions.    The concentration allows us to pair up maps with even ramification and show that {\em the contributions of the maps with even ramification cancel in pairs}.  This cancelation  is the key observation of the paper and is the final ingredient in the proof of  Theorem~\ref{Main}.

 \smallskip

Section~\ref{section11} presents some specific calculations:  Theorem~\ref{Main} is used to determine all   spin Hurwitz numbers with  degree $d= 4$  for every genus.

 \medskip

Very recently, S.~Gunningham \cite{G} has used completely different methods to obtain results that  overlap ours.  His approach  casts the spin Hurwitz numbers as a topological quantum field theory. He determined all spin Hurwitz numbers  (including etale spin Hurwitz numbers) in terms of the combinatorics of Sergeev algebras.  The exact relationship between  Gunningham's  results  and ours is not immediately clear.

\vspace{.6cm}


\setcounter{equation}{0}
\section{Antilinear deformations of $\del$}
\label{section2}
\bigskip

Let $f:C\to D$ be a  holomorphic map of  degree $d>0$ between smooth curves and let $N\to D$ be a theta characteristic. As shown in  \cite{LP1},  there is a holomorphic 2-form  on the total space of $N$ that induces a  conjugate-linear bundle map $R: f^*N \to \overline{K}_C \otimes f^*N$ with several remarkable properties. In this section we use a different approach to produce a similar bundle map $R: N_f \to \overline{K}_C\otimes N_f$ where $N_f$ is the twisted pullback bundle  (\ref{thetaonC}). This  map $R$ and the associated deformations $\del+tR$ of the the $\del$-operator on $N_f$ are the  central objects in this paper.

\medskip

 To start, let $D$ be a smooth curve with canonical bundle $K_D$ and a Riemannian metric.   Let $N\to D$ be a holomorphic line   bundle with a hermitian metric $\langle\ ,\ \rangle$, conjugate linear in the second factor. Let $( \ , \ )=\Re \langle\ ,\ \rangle$ be the corresponding positive definite inner product, and let $\bar{*}:\Lambda^{p,q}D\otimes N^*\to \Lambda^{1-p, 1-q}D\otimes N$ be  the associated conjugate-linear star operator.

 \begin{lemma}
 \label{Rlemma}
 Any holomorphic section $\varphi$ of $K_D\otimes (N^*)^2$ induces a bundle map
\begin{equation}
\label{2.defofReq}
 R:N\to \overline{K}_D\otimes N
\end{equation}
that, with its adjoint $R^*$ with respect to the  inner product $(\ ,\ )$,  satisfies
 \begin{equation}
 \label{2.threeproperties}
(a)\ \ R J= -JR \hspace{1cm}
(b)\ \  R^*R = |\varphi|^2\, Id   \hspace{1cm}
(c)\ \   \del^* R+R^*\del=0.
\end{equation}
 \end{lemma}

 \begin{proof}
Regard $\varphi$ as a complex bundle map $\varphi:N\to K_D\otimes N^*$ and set  $R=\bar{*}\circ\varphi$.  Because $\bar{*}$ is conjugate-linear, this  immediately gives $RJ=-J R$.    Fix a point $p$,  a local holomorphic coordinate $z$ around $p$  in which the metric is Euclidean to second order, and a local holomorphic section  $\nu$ of $N$ with $|\nu(p)|=1$.  It suffices to verify that (b) and (c) hold at $p$.

  Let $\nu^*$ denote the dual section to $\nu$ and write $\varphi(\nu) = g dz\nu^*$.  Then $\del g=0$ because $\varphi$ is holomorphic, and at  $p$  we have $R(\nu)= \bar{g}d\bar{z}\nu$ and, taking the adjoint using the real inner product,  $R^*(d\bar{z}\nu)= \bar{g}\nu$.  It follows that $R^*R=|g|^2\, Id = |\varphi|^2\, Id$.
Choosing an arbitrary section $\xi= f \nu$ of $N$, we have
\begin{equation}
\label{R*delequation}
R^*\del\xi \ =\  R^*\left(\tfrac{\partial f}{\partial{\bar{z}}}\, d\bar{z}\nu\right)
\ =\  \overline{\left(\tfrac{\partial f}{\partial{\bar{z}}}\right)} \, R^*(d\bar{z}\nu)
\ =\  \bar{g}\tfrac{\partial \bar{f}}{\partial{z}}\, \nu
\end{equation}
at $p$.   On the other hand,   the formulas $\del^*=- \bar{*}\del \bar{*}$ and $(\bar{*})^2=-1$ on $\Omega^{0,1}(N^*)$ show that at $p$
$$
\del^*R\xi
\ =\  -\bar{*}\, \del\,  \bar{*}\bar{*}\,\varphi(f\nu)
\ =\  \bar{*}\,\del\left(f  g\,   dz\nu^*\right)
\ =\ \bar{g} \tfrac{\partial \bar{f}}{\partial{z}}\ \bar{*} \,(d\bar{z}dz\nu^*)
\ =\  -\bar{g} \tfrac{\partial \bar{f}}{\partial{z}}\ \nu.
$$
 This  cancels (\ref{R*delequation}), giving statement (c).
 \end{proof}

An endomorphism $R$ as in (\ref{2.defofReq}) determines a  1-parameter family of perturbations of the $\del$-operator, namely the
 operators  $L_t:\Omega^0(N)\to \Omega^{0,1}(N)$ defined by
 \begin{equation}
 \label{Defdelt}
L_t=\del +t R \qquad  t\in\R.
  \end{equation}
Properties (\ref{2.threeproperties})  imply a remarkably simple vanishing theorem.

 \begin{VanishingTheorem}
 \label{2.vanishingtheorem}
If $R$  satisfies (\ref{2.threeproperties}) with $\varphi\not\equiv  0$, then $\ker L_t=0$ for each  $t\not= 0$.
 \end{VanishingTheorem}

 \begin{proof}
If $L_t\psi=0$ then by (\ref{2.threeproperties}) we have
\begin{equation}
\label{2.vanishintegral}
0\ =\ \int_D | L_t\xi|^2\ =\
\int_D  |\del\xi|^2\ +\ |R|^2\, |\xi|^2.
\end{equation}
Thus $\xi$ is holomorphic and vanishes on the open set where $R=\varphi\not= 0$, so $\xi\equiv 0$.
 \end{proof}

\bigskip

 Many of the results in subsequent sections can be viewed as natural extensions of  Theorem~\ref{2.vanishingtheorem}.  For a first extension, let   $f:C\to D$ be a holomorphic map  with ramification points $q_j$ and  ramification divisor $\Rf=\sum (m_j-1)q_j$.  If $N$ is a   theta characteristic  on $D$ and $A$ is any divisor on $C$, we can consider the twisted bundle
 \begin{equation*}
 \label{1.2}
N_f\ =\ f^*N\otimes \O_C(A)
\end{equation*}
on $C$. We then have the following  relative version of Lemma~\ref{Rlemma}.

\begin{cor}
\label{2.tRfcor}
A holomorphic section $\varphi$ of $\O(\Rf-2A)$ induces a bundle map
\begin{equation}
\label{defofReq}
 R_f:N_f\to \overline{K}_C\otimes N_f
\end{equation}
that satisfies properties (\ref{2.threeproperties}), and the Vanishing Theorem~\ref{2.vanishingtheorem} applies to $L_t=\del+tR_f$.
\end{cor}

\begin{proof}
The Hurwitz formula and the  isomorphism $N^2= K_D$  induce  an isomorphism $K_C\otimes (N_f^{*})^2   = \O(\Rf-2A)$, so we can apply  Lemma~\ref{Rlemma}.
\end{proof}

\vspace{.6cm}


\setcounter{equation}{0}
\section{Parity as the TR spectral flow}
\label{section3}
\bigskip

Suppose that $A_t:V_t\to W_t$ is a smooth path of linear maps where $V_t$ and $W_t$ are the fibers of real vector bundles  $V$ and $W$ over $\R$.   The real variety ${\cal S}\subset Hom(V, W)$ of non-invertible maps separates  the bundle $Hom(V,W)$ into connected open sets called chambers. If $A_{t_1}$ and $A_{t_2}$ are non-singular, the {\em mod 2 spectral flow} of the path $A_t$ from $t_1$ to $t_2$   is calculated by perturbing the family to be transversal to ${\cal S}$ and counting the number of times the family crosses ${\cal S}$ modulo 2; this is independent of the perturbation.    This section describes  a modified spectral flow that applies to the operators  $L_t=\del +t R$ of (\ref{defofReq}).

\medskip

We begin with a definition that occurs in quantum mechanics.   Let $V$ and $W$ be real vector bundles over $\R$.
 A  {\em TR (``time-reversal'') structure} is a lift of the map $t\mapsto -t$ to bundle maps $T: V\to V$ and  $T:W\to W$ satisfying $T^2=-Id.$  A bundle map $A:V\to W$  is  {\em TR invariant} if there is a $T$ as above such that
\begin{equation}
\label{3.New.1}
[ A, T]=0 \qquad \quad \mbox{that is, }\ \  A_{-t}\ =\ T_t A_t T_t^{-1}.
\end{equation}
In particular, $T_0=J$ is  a complex structure  on  $V_0$ and $W_0$ and  by (\ref{3.New.1}) and $A_0$ is  complex linear.  

Let  ${\cal TR}$ denote the space of all smooth TR invariant $A:V\to W$ that are invertible  except  at finitely many values of $t$.  For an open dense set of  $A\in {\cal TR}$,    $A_0$ is non-singular and $A$ intersects ${\cal S}$ transversally at finitely many points $\{\pm t_i\}$ (proof:  given $A$, perturb $A_t$ for $t\geq 0$ to $A'_t$ with  transversal to ${\cal S}$ and  $A'_0$ is complex and invertible, then define $A_{-t}$ by (\ref{3.New.1}) and  smoothing, symmetrically in $t$,  around $t=0$). Thus the  mod 2 spectral flow from $t=-\infty$ to $t=\infty$ is well-defined, but is 0 because the singular points are symmetric.
However, there is  a well-defined   {\em TR spectral flow} 
\begin{equation}
\label{3.defTRSF}
SF^{TR}:  {\cal TR} \to \{\pm 1\}
\end{equation}
defined for $A\in {\cal TR}$ by perturbing to a generic $C\in {\cal TR}$ and setting $SF^{TR}(A)= (-1)^s$ where $s$ is  the mod 2 spectral flow of $C$  from $t=0$ to $t=\infty$.   Regarding $C$  as a path in $Hom(V,W)$,  $s$ is the mod 2 intersection $C\cap {\cal S}$, which depends only on the homology class is $C$.  If $D$  is another generic perturbation with $s'=D\cap{\cal S}$, then $s-s' = \gamma\cap {\cal S}$ where $\gamma$ is a path from $B_0$ to $C_0$.  But then $s-s'$ is even because $B_0$ to $C_0$ are complex -linear isomorphisms.  Thus $SF^{TR}(A)$ is independent of the perturbation.   In practice, two  formulas are useful:

\medskip

(i)   If $V$ and $W$ both have finite rank $r$,  the complex orientation on $V_0$ and $W_0$ extends   to orient all fibers of  $V$ and $W$.    This means that $\sgn \det A_t$ is canonically defined for every $A\in{\cal TR}$ and all $t$.  For generic $A\in {\cal TR}$  the sign of $\det A_t$ is positive for $t=0$ and changes sign with each transversal crossing of ${\cal S}$.  Thus
\begin{equation}
\label{3.SFformula1}
SF^{TR}(A)\ =\ \sgn \det A_t
\end{equation}
whenever $A_s$ is non-singular for all $s\geq t$.

\smallskip

(ii) Now suppose that $\ker A_0$ is finite-dimensional and $\dot{A}_0$ restricts to an isomorphism $B:\ker A_0\to \cok A_0$.  Then $\ker A_0$ and $\cok A_0$ are complex vector spaces of the same dimension $d$.   Choose a complex-linear map $C:\ker A_0\to \cok A_0$ and perturb
 $A$ to a generic  $A'\in{\cal TR}$ as above with $A'_0=A_0+\ep C$ and $A'(t)=A_t$ for all $t\geq \delta$.  Then $\det C>0$ and the mod 2 spectral flow of $A'_t$ from $t=0$ to $t=\delta$ is $\sgn\det B$.  But one sees by differentiating  (\ref{3.New.1}) that $B$ satisfies $JB=-BJ$; therefore $\det B=(-1)^d$ because the    eigenvalues of $B$ come in pairs $\pm \sqrt{-1}\, \la_i$.  We conclude that
\begin{equation}
\label{3.SFformula2}
SF^{TR}(A)\ =\   (-1)^{\mbox{$\mbox{dim}_{\cx}\, \ker A_0$}} 
\end{equation}

\bigskip

The TR spectral flow readily applies to the operators introduced in Section~\ref{section3}.   Let $(D, N)$ be a  smooth spin curve with a  bundle map  $R$  as in Lemma~\ref{Rlemma} that is non-zero almost everywhere.  For each $t$,  $L_t=\del +tR$ extends to a  Fredholm map
$$
L_t: V_\cx \to W_\cx
$$
from the Sobolev $W^{1,2}$  completion  of $\Omega^0(N)$  to the $L^2$ completion  of $\Omega^{0,1}(N)$.  By elliptic theory, $V_\cx$ (resp. $W_\cx$)  decomposes into finite-dimensional real eigenspaces $E_\la$ of  $L_t^*L_t$ (resp. $L_tL^*_t$) whose  eigenvalues $\{\la\}$  are real, non-negative, discrete, and vary continuously with $t$. For each $t$, let $V_t\subset V_\cx$ and $W_t\subset W_\cx$ be the closure of the {\em real} span of the eigenspaces; these form vector bundles $V, W$ over $\R$.  By Property (\ref{2.threeproperties}a) we have
\begin{equation}
\label{3.JLJL}
JL_tJ^{-1}\ =\ -J(\del +tR)J\ =\ -J (J\del-tJR)\ =\ L_{-t}.
\end{equation}
Thus $T=J$ is a TR structure and $L=\{L_t\}$ is a TR-invariant operator.

To calculate the invariant (\ref{3.defTRSF}), we can reduce to a finite-dimensional situation. Fix $\la_0>0$ not in the spectrum of $\del^*\del$ and    define the low eigenspaces  of $L_t^*L_t$ and $L_tL_t^*$ by setting
\begin{equation}
\label{3.defEF}
E_t =  \bigoplus_{\la<\la_0} E_\la
\ \ \ \ \  \mbox{and}\ \ \ \ \
F_t =  \bigoplus_{\la<\la_0} F_\la.
\end{equation}
These form finite-rank real vector bundles $E\subset V$ and $F\subset W$ over an interval $[-\delta, \delta]$ where $\la_0$ remains outside the spectrum, and (\ref{3.JLJL}) again shows that $L:E\to F$ is TR-invariant.

\begin{theorem}
\label{Injective-Esp}
The parity of a spin structure $(D,N)$ is the TR spectral flow of the Fredholm operator $L:V\to W$, and for $0<|t|<\delta$  it   is also  the determinant of the finite-dimensional operator $L_t: E_t\to F_t$:
$$
p\ =\ SF^{TR}(L)
\qquad\mbox{and}\qquad
SF^{TR}(L)\ =\  \sgn \det L_t \qquad \mbox{for $|t|\leq \delta$.}
$$
\end{theorem}

\begin{proof}
By its definition (\ref{introparity}), the parity is $p=(-1)^h$ where $h=\dimc \ker \del = \dimc \ker L_0$.  Observe that  $\dot{L}_0=R$ is injective on $\ker \del$ by Theorem~\ref{2.vanishingtheorem}, and hence is an isomorphism because Riemann-Roch shows that $\dim \ker \del-\dim \cok \del = \chi(D, N)=0$.  The
first equality therefore follows by  (\ref{3.SFformula2}).  For  all $-\delta\leq t\leq \delta$, $L_t^*L_t$ is non-singular on the eigenspaces with   $\la>\la_0$, so the spectral flow is determined by the restriction of $L_t$ to the low eigenspaces (\ref{3.defEF}), where it is given by formula (\ref{3.SFformula1}).
\end{proof}

 As a corollary, we obtain  a simple proof of the Atiyah-Mumford Theorem on spin structures.
 
 \begin{cor}
\label{A-Mtheorem}
Parity is an isotopy invariant of spin structures $(D,N)$.
\end{cor}

 \begin{proof}
 If $(D_s, N_s)$ is a path of spin curves then $K_{D_s}(N_s^*)^2=\O$ is trivial for each $s$, so there are smoothly varying nowhere-zero maps $R_s$ as in  Lemma~\ref{Rlemma}.  For fixed $t\not= 0$,   Theorem~\ref{2.vanishingtheorem}   shows that  $L_{s}=\del+tR_s$ is injective for all $s$, so $SF^{TR}(L_{s})$  -- and hence the parity -- is independent of $s$.
 \end{proof}
 
 In Sections~\ref{section5}-\ref{section8} we will extend this proof by incorporating maps as in Corollary~\ref{2.tRfcor} and applying it to families of spin curves that degenerate to nodal curves.

\vspace{.6cm}

\setcounter{equation}{0}
\section{Degeneration, gluing and the proof of Theorem~\ref{Main} }
\label{section4}
\bigskip

The proof of Theorem~\ref{Main} is based on the 
method of \cite{L}:
we express the  spin Hurwitz numbers in terms of  relative Gromov-Witten moduli space and
apply the limiting and gluing arguments of \cite{IP2} for a degeneration of spin curves
to form a family of moduli spaces.
We  then use a smooth model of the family of moduli spaces to calculate  parities.
The calculation immediately yields the desired recursion formula.
This section outlines the proof, drawing on two facts that are deferred:
the construction of a smooth model (done in Sections~\ref{section5} and \ref{section6}), and the computation of parities (done in Sections~\ref{section8}-\ref{section10}).

As in \cite{L}, we begin by  expressing the spin Hurwitz numbers (\ref{SHN}) in terms of GW relative moduli spaces (cf. \cite{IP1}).
Let $D$ be a smooth curve of genus $h$ and let  $V=\{q^1,\cdots,q^k\}$ be a fixed set of points on $D$.
Given partitions $m^1,\cdots,m^k$ of $d$,
a degree $d$ holomorphic map $f:C\to D$ from a (possibly disconnected)  curve $C$ is called
{\em $V$-regular} with  contact partitions $m^1,\cdots,m^k$ if, for each $i=1,\dots k$, $f^{-1}(q^i)$ consists of $\ell(m^i)$ points $q^i_j$ so that the ramification index of $f$ at $q^i_j$ is $m^i_j$. If $m^i_j>1$ then the contact marked point $q^i_j$ is a ramification point of $f$
and $q^i$ is a branch point.
The relative moduli space
\begin{equation}\label{RelativeModuli}
\M^{V}_{\chi,m^1,\cdots,m^k}(D,d)
\end{equation}
consists of isomorphism classes of   $V$-regular maps $(f,C;\{q^i_j\})$ with  contact vectors $m^1,\cdots,m^k$. Here
$\chi(C)=\chi$ and all marked points are contact marked points.
Since no confusion can arise,
we will often  write  $(f,C;\{q^i_j\})$ simply as $f$.

Spin Hurwitz numbers are associated with those moduli spaces (\ref{RelativeModuli}) that have (formal) dimension 0.   Thus we will henceforth assume that
\begin{equation}\label{DimRelModuli}
\mbox{dim}_\cx\, \M^{V}_{\chi,m^1,\cdots,m^k}(D,d)\ =\
2d(1-h) -\chi  - \sum_{i=1}^k \big(\,d-\ell(m^i)\,\big)\ =\ 0.
\end{equation}
With this assumption,  all ramification points of a $V$-regular map $(f,C;\{q^i_j\})$ in (\ref{RelativeModuli}) are contact marked points.
In this case, forgetting the contact marked points  gives  a (ramified) covers $f$ satisfying (\ref{introdim0}).
If $m^i=(1^d)$ for some $1\leq i\leq k$ then
\begin{equation}\label{SHbyRelModuli}
H^{h,p}_{m^1,\cdots,m^k}\ =\ \frac{1}{\prod m^i!}\,\sum \,p(f) 
\end{equation}
the sum is over all  $f$ in (\ref{RelativeModuli}) and
$p(f)$ is the associated parity (\ref{introparity}) (cf. Lemma~1.1 of \cite{L}).

Adding trivial partitions does not change the formulas (\ref{introdim0}) and (\ref{DimRelModuli}). It also does not change
the spin Hurwitz numbers, namely,
\begin{equation}\label{trivial}
H^{h,p}_{(1^d),m^1,\cdots,m^k}\ =\ H^{h,p}_{m^1,\cdots,m^k}.
\end{equation}
Below, we fix $h$, $d$, $\chi$ and odd partitions $m^1,\cdots,m^k$ of $d$
so that the dimension formula (\ref{DimRelModuli}) holds.
In light of (\ref{SHbyRelModuli}), we will  add trivial partitions $m^{k+1}=m^{k+2}=m^{k+3}=(1^d)$ to make our discussion simpler.

\medskip

To adapt the main argument of \cite{IP2} we will build a  degeneration of target curves. Let $D_0=D_1\cup E \cup D_2$ be a connected nodal curve of arithmetic genus $h$ of a rational curve $E$ and smooth curves $D_1$  and $D_2$ of genus $h_1$ and $h_2$ with $h_1+h_2=h$,  joined at nodes $p^1=D_1\cap E$ and $p^2=D_2\cap E$. Fix points $k+3$ points $q^i$, all distinct and distinct from $p^1$ and $p^2$, with
\begin{equation*}\label{3.qpoints}
q^{k+1},q^1,\cdots,q^{k_0}\in  D_1, \qquad
q^{k+2}\in E, \qquad
q^{k_0+1},\cdots,q^k,q^{k+3}\in   D_2.
\end{equation*}
 where $0\leq k_0\leq k$.
In Section~\ref{section5}, we will construct a deformation of $D_0$  with sections;  it is a smooth complex surface $\D$ fibered over the disk $\Delta$ with parameter $r$
\begin{equation}\label{deg-target}
\xymatrix{
 \D \ar[d]^{\rho} \\
 \Delta \ar@/^1pc/[u]^{Q^i}    }
\end{equation}
so that the central fiber is $D_0$, the   fibers $D_r$ with $r\ne 0$ are smooth curves of genus $h$ and $Q^i(0)=q^i$ for  $1\leq i\leq k+3$.

For each partition $m$ of $d$,  consider the moduli  space of maps
\begin{equation}\label{central-moduli}
\CP_m = \M^{V_1}_{\chi_1, m^{k+1},m^1,\cdots,m^{k_0},m}(D_1,d) \times  \M^{V_e}_{\chi_e, m,m^{k+2},m}(E,d) \times  \M^{V_2}_{\chi_2, m,m^{k_0+1},\cdots,m^k,m^{k+3}}(D_2,d)
\end{equation}
where $V_1=\{q^{k+1},q^1,\cdots,q^{k_0},p^1\}$, $V_e=\{p^1,q^{k+2},p^2\}$, $V_2=\{p^2,q^{k_0+1},\cdots,q^k,q^{k+2}\}$ and
\begin{equation}\label{SumEulerC}
\chi_1+\chi_e+\chi_2-4\ell(m)\ =\ \chi.
\end{equation}
For simplicity, let $\M^1_m,\M_m^e$ and $\M_m^2$ denote the first, second and third factors of $\CP_m$.
By (\ref{SumEulerC}) and our assumption that the  dimension formula (\ref{DimRelModuli}) holds, it is easy to see that whenever
 the  space $\CP_m$ is not empty,   the relative moduli spaces  $\M^1_m,\M_m^e$ and $\M_m^2$
all have dimension zero. In particular, $\chi_e=2\ell(m)$ and
\begin{equation}\label{NumM0}
|\M_m^e|\ =\ \frac{d! \,m!}{|m|}
\end{equation}
where $|\M_m^e|$ denotes the cardinality of  $\M_m^e$ (cf. Section 2 of \cite{L}).

For $(f_1,f_e,f_2) \in \CP_m$, let $x^i_j$ and $y^i_j$ be contact marked points
of $f_i$ and $f_e$ over $p^i\in D_i\cap E$ with multiplicity $m_j$ where $i=1,2$ and $j=1,\cdots,\ell(m)$. By identifying $x^i_j$ with $y^i_j$,
one can glue the domains of $f_i$ and $f_e$ to obtain a map $f:C\to D_0$ with $\chi(C)=\chi$.
For notational convenience, we will often write the glued map $f$ as $f=(f_1,f_e,f_2)$.
Denote by
\begin{equation}\label{CenterModuli}
\M_{m,0}
\end{equation}
the space of such maps $f=(f_1,f_e,f_2)$.
Identifying contact marked points associates to each node of $C$ a multiplicity $m_j$ labeled by $j$.
But the nodes of $C$ are not labeled.  One can thus  see that
gluing domains gives a degree $(m!)^2$ covering map:
\begin{equation}\label{GluingMap}
\CP_m\ \to\ \ \M_{m,0}.
\end{equation}

\begin{remark}\label{RamificationIndex}
Let $f=(f_1,f_e,f_2)$ be a map in $\M_{m,0}$. For $i=1,2$, since $\M^i_m$ has dimension zero,
(i) the  ramification points of $f_i$ are either contact marked points or nodal points of the domain of $f$,
(ii) $f_i$ can have even ramification points only at nodal points and
(iii) the number of even ramification points of $f_i$ is even.
\end{remark}

For $r\not= 0$, consider the  moduli spaces of $V$-regular  maps into $D_r$, which we denote by
\begin{equation}\label{def-Mr}
\M_r\ =\ \M^{V_r}_{\chi,m^1,\cdots,m^{k+3}}(D_r,d)\ \ \ \ \
\mbox{where}\ \ \ \ \ V_r=\{Q^1(r),\cdots,Q^{k+3}(r)\},
\end{equation}
By Gromov convergence, a sequence of holomorphic maps into $D_r$ with $r\to 0$ limits has a subsequence that converges to a  map into $D_0$.
Denote the set of such  limits by
\begin{equation*}\label{LimitSet}
\lim_{r\to 0} \M_r.
\end{equation*}
Lemma~3.1 of \cite{L} shows that
\begin{equation}\label{LimitModuli}
\lim_{r\to 0}\ \M_r \ \subset \ \bigcup_m\  \M_{m,0}
\end{equation}
where the union is over all partitions $m$ of $d$ with $\CP_m\ne \emptyset$.


Conversely, by the Gluing Theorem of \cite{IP2}, the domain of each map in $\M_{m,0}$ can be smoothed to produce
maps in $\M_r$ for small $|r|$.
Shrinking $\Delta$ if necessary, for $r\in \Delta$,   one can assign
to each $f_r\in \M_r$ a partition $m$ of $d$ by (\ref{LimitModuli}).
Let $\M_{m,r}$  be the set of all  pairs $(f_r,m)$ and for each $f\in\M_{m,0}$ denote by
\begin{equation}\label{ComFModuli}
\CZ_{m,f}\ \to\ \Delta
\end{equation}
the connected component of $\bigcup_{r\in\Delta} \M_{m,r}\to \Delta$ that contains $f$.
It follows that
\begin{equation}\label{DecMr}
\M_r\ =\ \bigsqcup_m \bigsqcup_{f_r\in \M_{m,0}} \CZ_{m,f,r}\ \ \ \ \ \ (r\ne 0)
\end{equation}
where $\CZ_{m,f, r}$ is the fiber of (\ref{ComFModuli}) over $r\in \Delta$.
The Gluing Theorem shows that one can smooth each node $x^i_j=y^i_j$ of the domain of $f$, where $i=1,2$ and $j=1,\cdots,\ell(m)$,
in $m_j$ ways to produce $|m|^2$ maps in $\M_{m,r}$, i.e., the fiber $\CZ_{m,f, r}$ ($r\ne 0$) consists of $|m|^2$ maps.

\bigskip
We now introduce a spin structure on $\rho:\D\to\Delta$ assuming that $\D$ is smooth.  Given parities $p$, $p_1$ and $p_2$ with $p_1+p_2=p$ (mod 2),
 Cornabla's  \cite{C} constructs  a
line bundle $\L\to\D$ and  a homomorphism
\begin{equation}
\label{homomorphism}
\Phi:\L^2\to K_\D
\end{equation}
whose restrictions satisfy the following properties:
\begin{itemize}
\item[(a)]
For $r\ne 0$,  $\L$ restricts to a  theta characteristic  on $D_r$ with a parity $p$ and  $\Phi$ restricts to an isomorphism $(\L|_{D_r})^2\to K_{D_r}$.
\item[(b)]
$\Phi$ vanishes identically on $E$ and $\L|_E =\O_E(1)$.
\item[(c)]
For $i=1,2$, $\L$ restricts to a  theta characteristic  on $D_i$ with  parity $p_i$,  and   $\Phi$ restricts to an isomorphism  $(\L|_{D_i})^2\to K_{D_i}$.
\end{itemize}
The pair $(\L,\Phi)$ is called a spin structure on $\rho:\D\to\Delta$.

Let $f=(f_1,f_e,f_2)$ be a map in $\M_{m,0}$.
Note that all ramification points of  maps  in $\CZ_{m,f,r}$  ($r\ne 0$) have odd ramification indices
since $m^1,\cdots,m^k$ are odd partitions. So, each map $f_r$  in $\CZ_{m,f,r}$ has an associated parity $p(f_r)$
 defined as in (\ref{introparity}) by the pull-back bundle $f_r^*(\L|_{D_r})$ and its ramification divisor ${\cal R}_{f_r}$.
When the partition $m$ is  odd, $f_i$ ($i=1,2$) also have  associated parities $p(f_i)$ defined by
$f_i^*(\L|_{D_i})$ and ${\cal R}_{f_i}$.  In this context,  (\ref{SHbyRelModuli}), (\ref{trivial}) and (\ref{DecMr}) shows that for
$r\ne 0$  we have
\begin{equation}\label{MThmE1}
H_{m^1,\ldots,m^k}^{h,p}\ =\ H_{m^1,\ldots,m^k,(1^d),(1^d),(1^d)}^{h,p}\ =\
\frac{1}{(d!)^3\prod_{i=1}^k m^i!}\,\sum_m \, \sum_{f\in \M_{m,0}} \,\sum_{f_r\in \CZ_{m,f,r}}\,p(f_r).
\end{equation}

In Sections~\ref{section5}-\ref{section10} we will establish the following facts about the parity.

\begin{theorem}\label{MainTask}
Let $f=(f_1,f_e,f_2)\in \M_{m,0}$ and $r\ne 0$.
\begin{itemize}
\item[(a)] If $m$ is odd, then $p(f_r)=p(f_1)\,p(f_2)$ for all $f_r\in \CZ_{m,f,r}$.
\item[(b)] If $m$ is even, then ${\displaystyle \sum_{f_r\in \CZ_{m,f,r}}\,p(f_r) =0}$.
\end{itemize}
\end{theorem}

We conclude this section by showing how  Theorem~\ref{Main}a  follows from
Theorem~\ref{MainTask}.

\medskip
\non
{\em Proof of Theorem~\ref{Main}a:}
Together with  (\ref{MThmE1}), Theorem~\ref{MainTask} shows
\begin{equation}\label{MThmE2}
H_{m^1,\ldots,m^k}^{h,p}\ =\
\frac{1}{(d!)^3\prod_{i=1}^k m^i!}\,\sum_{m:odd} \, |m|^2\,\sum_{f=(f_1,f_e,f_2)\in \M_{m,0}} \hspace{-1.5em} p(f_1)\, p(f_2)
\end{equation}
where the factor $|m|^2$ appears because the fiber $\CZ_{m,f, r}$ ($r\ne 0$) consists of $|m|^2$ maps.
Since the map (\ref{GluingMap}) has degree $(m!)^2$, the last sum in (\ref{MThmE2}) is
\begin{align}\label{MThmE3}
\sum_{f=(f_1,f_e,f_2)\in \M_{m,0}} \hspace{-1.5em}p(f_1)\,p(f_2)\
&=\
\frac{1}{(m!)^2}\,\sum_{(f_1,f_e,f_2)\in \CP_m }   \hspace{-1.5em}  p(f_1)\,p(f_2) \notag  \\
&=\
\frac{1}{(m!)^2}\,\sum_{f_e\in \M_m^e}\Big(\,\sum_{f_1\in\M_m^1} p(f_1)\,\Big)\cdot \Big(\,\sum_{f_2\in\M_m^2} p(f_2)\,\Big)
\notag \\
&=\
\frac{(d!)^3 m!}{|m|}\,\prod_{i=1}^{k} m^i\cdot H^{h_1,p_1}_{m^1,\cdots,m^{k_0},m} \cdot H^{h_2,p_2}_{m,m^{k_0+1},\cdots,m^k}
\end{align}
where the last equality holds by (\ref{SHbyRelModuli}) and (\ref{NumM0}).  Theorem~\ref{Main}a follows from equations
(\ref{MThmE2}) and (\ref{MThmE3}).

\medskip

The proof of Theorem~\ref{Main}b is identical to that of Theorem~\ref{Main}a except that one uses
 a smooth  family of target curves $\D\to \Delta$  and a line bundle $\L\to \D$
satisfying:
\begin{itemize}
\item
The general fiber $D_r$ ($r\ne 0$) is a smooth curve of genus $h\geq 1$ and $\L|_{D_r}$ is a theta characteristic.
\item
The central fiber of $\D\to \Delta$ is a connected nodal curve $\bar{D}\cup E$ where  $\bar{D}$ is a smooth genus $h-1$ curve that meets
$E\cong {\mathbb P}^1$ at two points.
\item
 $\L$ restricts to $\O(1)$ on $E$ and to a theta characteristic  on $\bar{D}$  with
$p(\L|_{\bar{D}})\equiv p(\L|_{D_r})$ for $r\ne 0$.
\end{itemize}
Minor modifications to  the arguments of this section and to the constructions and calculations in Sections~\ref{section5}-\ref{section10},
yield parity formulas analogous to Theorem~\ref{MainTask}, which leads to Theorem~\ref{Main}b.
\qed

\vskip 1cm


\setcounter{equation}{0}
\section{The algebraic family moduli space}
\label{section5}
\bigskip

 In this section we  construct
  a  deformation of a map $f:C\to D_0$  from a nodal  curve to a nodal spin curve.
  The deformation has many components, indexed by roots of unity.
   Each component is a curve $\C\to \Delta$ over the disk with smooth total space,
   with a map to a deformation $\D\to\Delta$ of $D_0$ and a bundle $\N\to\C$ whose restriction
   to each general fiber $C_s$ is a theta characteristic on $C_s$.  In fact, there are  many
  such   bundles $\N$;   we fix one that makes computations in later sections possible.

\medskip

Throughout this section we fix, once and for all, a partition $m=(m_1,\cdots,m_\ell)$ of $d$,
a map $f=(f_1,f_e,f_2):C\to D_0$ in $\M_{m,0}$ where $\M_{m,0}$ is the space (\ref{CenterModuli}), and the spin structure $(\L, \Phi)$ on $\rho:\D\to \Delta$  in (\ref{homomorphism}).
As  in Section~\ref{section4},
$D_0$ is a nodal curve $D_1 \cup E \cup D_2$ with exceptional component $E={\mathbb P}^1$  and with   nodes $p^1\in D_1\cap E$ and
$p^2\in D_2\cap E$.
The domain $C$ is a nodal curve $C=C_1\cup C_e\cup C_2$
with $2\ell$ nodes where $\chi(C_e)=2\ell$ such that for $i=1,2$ and $j=1,\cdots,\ell$
\begin{itemize}
\item
$f^{-1}(p^i)$   consists of the $\ell$ nodes $p^i_j\in C_i\cap E_j$,
\item
$C_i$ is smooth and $f_i=f|_{C_i}:C_i\to D_i$ has ramification index $m_j$ at the node $p^i_j$,
\item
$C_e$ is a disjoint union of $\ell$ rational curves $E_j$, $f_e=f|_{C_e}$ and
each restriction ${f}|_{E_j}:E_j\to E$ has degree $m_j$ and  ramification index $m_j$ at  $p^i_j$.
\end{itemize}
For $i=1,2$, let ${\cal R}_{f_i}$  denote ramification divisor of $f_i$, and
let ${\cal R}_{f_i}^{ev}$ be  the  divisor on $C_i$  consisting of the {\em even} ramification points:
\begin{equation}\label{SquareRam}
{\cal R}_{f_i}^{ev}\ =\ \sum_{j\, |\, \mbox{\scriptsize $m_j$ is even}} p^i_j.
\end{equation}
By Remark~\ref{RamificationIndex},  $|{\cal R}_{f_i}|$ and $|{\cal R}_{f_i}^{ev}|$ are both even. For $j=1,\cdots,\ell$, we  set
\begin{equation}\label{NumberOfchain}
n_j\ =\ \frac{|m|}{m_j}.
\end{equation}

 Let $Q_m$ denote the set of vectors of the form
$\zeta=(\zeta_1^1, \zeta_1^2, \cdots \zeta_\ell^1, \zeta_\ell^2)$
where $\zeta_j^1$ and $\zeta_j^2$ are $m_j$-th roots of unity.
The following is a main result of this section.

\begin{theorem}\label{TwistingProp}
 Let $f=(f_1,f_e,f_2)\in\M_{m,0}$ and $Q_m$ be as above.
Then, for each vector $\zeta\in Q_m$,
there exists a family of curves $\C_\zeta\to \Delta$,over a disk $\Delta$ (with parameter $s$) with smooth total space $\C_\zeta$,  a holomorphic map $\F_\zeta:\C_\zeta\to\D$  and a line bundle $\N_\zeta$ over $\C_\zeta$
satisfying:
\begin{itemize}
\item[(a)]
For $s\not= 0$, the  fiber $C_{\zeta,s}$  is smooth and the restriction $N_{\zeta,s}=\N_\zeta|_{C_{\zeta,s}}$ is a theta characteristic on $C_{\zeta,s}$ and the restriction map $f_{\zeta,s}=\F|_{C_{\zeta,s}}$ has the associated parity $p(f_{\zeta,s}) = p(N_{\zeta,s})$
such that the last sum in (\ref{MThmE1}) is
\begin{equation}\label{Summary-Shiffer}
\sum_{f_r\in \CZ_{m,f,r}}\!\!p(f_r)\ =\
\sum_{ \zeta\in Q_m} \ p(f_{\zeta,s})
\qquad \mbox{where $r=s^{|m|}$}.
\end{equation}

\item[(b)]
The central fiber  $C_{\zeta,0}$ is a nodal curve $C_1 \cup (\,\cup_{j=1}^\ell \bar{E}_j\, )\cup C_2$  where each $\bar{E}_j$ is a chain of rational curves with dual graph
{\small
\begin{equation}\label{Adiagram}
  \begin{tikzpicture}[start chain]
\LabeledCircle{}{$C_1$}
\LabeledCircle{}{\ \ $E^1_{j;n_j-1}$}
\DiagramDots
\LabeledCircle{}{$E_{j;1}^1$}
\LabeledCircle{}{$E_j$}
\LabeledCircle{}{$E_{j1}^{2}$}
\DiagramDots
\LabeledCircle{}{\ $E^2_{j; n_j-1}$}
\LabeledCircle{}{$C_2$}
\end{tikzpicture}
\end{equation}   }

\item[(c)]
$ \N_\zeta|_{C_i}= f_i^*(\L|_{D_i})\otimes\ \O\big(\,\tfrac12({\cal R}_{f_i}-{\cal R}_{f_i}^{ev})\,\big)$\ \ for $i=1,2$.
\item[(d)]  $ \N_\zeta|_{E^i_{j;k}} =
\begin{cases}
\O(1)  &\mbox{if\ $m_j$\ is\ even\ and\ $k=n_j-1$, \ and\ if\ $m_j$\ is\ odd\ and\ $k= 0$}\ \\
\O  &\mbox{otherwise.  } \\
\end{cases}$
\end{itemize}
Here, for the case $k=0$,  $E^1_{j;0}=E^2_{j;0}$ denotes $E_j$.  Note that $n_j>1$ whenever $m_j$ is even (because $|{\cal R}_{f_i}^{ev}|$ is even).
\end{theorem}

\bigskip

The proof of Theorem~\ref{TwistingProp} requires  6 steps; each is a standard procedure in algebraic geometry.  Steps~1-4 use  Shiffer Variations  (cf. \cite{ACGH})  and  are described in detail in \cite{L}.

\medskip
\non
{\bf Step 1 -- Deform the target:}  As in (\ref{deg-target}) there is an algebraic curve
 $\rho:\D\to \Delta$ over the disk $\Delta$ with $k+3$ sections $Q^i$ whose central fiber  is identified with $D_0$ with the marked points $q^i=Q^i(0)$.  Denoting the coordinate on $\Delta$ by $r$, there are local coordinates $(u^1, v^1, r)$ and $(u^2, v^2, r)$ around the nodes $p^1$ and $p^2$ in $\D$ so that the fiber $D_r=\rho^{-1}(r)$ is locally given by $u^1v^1=r$ and $u^2v^2=r$.

\bigskip
\non
{\bf Step 2 -- Deform the domain:}  A similar construction yields a  deformation $\varphi_{2\ell}: \X \ \to\  \Delta_{2\ell}$ of $C_0$ over  polydisk
\begin{equation}\label{def-domain}
\Delta_{2\ell}\,=\,\{\,\r=(r^1_1,r^2_1,\cdots,r^1_\ell,r^2_\ell)\in\cx^{2\ell} \,:\,|r^i_j|<1\,\}.
\end{equation}
Furthermore,  there are local coordinates $(x^i_j,y^i_j,\r)$ around each node $p^i_j$ of $C_0$ in $\C$  in which the fiber $C_{\bf r}$ of $\rho$ over $r$ is given by $x^i_jy^i_j=r^i_j$.

\bigskip
\non
{\bf Step 3 -- Extend the map:}\ \ The map $f:C\to D_0$ extends to a map of families over  the curve ${\cal V}\subset \Delta_{2\ell}$ defined by
\begin{equation}\label{matching}
{\cal V}\ =\ \Big\{(r^1_1)^{m_1} = (r^2_1)^{m_1} = \ \cdots\ =(r^1_\ell)^{m_\ell} = (r^2_\ell)^{m_\ell} =r\ \Big|\ r\in\cx\Big\}.
\end{equation}
Near the nodes of $C_0$ the extension is given on $\varphi_{2\ell}^{-1}({\cal V})$  by
\begin{equation}
\label{LocalMap2}
(x^i_j,y^i_j,\r)\ \to\ (u^i, v^i, r) \qquad\mbox{where}\  u^i=(x^i_j)^{m_j}, v^i=(y^i_j)^{m_j},  r=(r^i_j)^{m_j}.
\end{equation}
  Note that this extension maps fibers to fibers only over  ${\cal V}$.

\bigskip
\non
{\bf Step 4 -- Normalization:}\ \   The one-dimensional variety (\ref{matching})  has  $|m|^2$ branches at the origin.
To separate the branches, we lift to the normalization  as follows.   For each vector  $\zeta=(\zeta_1^1,\zeta_1^2,\cdots,\zeta_\ell^1,\zeta_\ell^2)$ in  $Q_m$,
define a holomorphic map
\begin{equation*}\label{zeta-map}
\delta_\zeta : \Delta\to \Delta_{2\ell}\ \ \ \ \
\mbox{by}\ \ \ \ \
s\ \to\ (\zeta_1^1s^{n_1},\zeta^2_1s^{n_1}, \zeta_2^1s^{n_2},\zeta^2_2s^{n_2},\cdots,
\zeta_\ell^1 s^{n_\ell},\zeta_\ell^2 s^{n_\ell})
\end{equation*}
where $n_j$ is the number (\ref{NumberOfchain}).
The pull-back $\X_\zeta=\delta_\zeta^*\X$ is  a deformation
of $C$ over  $\Delta$:
$$
\xymatrix{
 \X_\zeta \ar[rr]  \ar[d]_{\wh{\varphi}_\zeta}
 &&  \X  \ar[d]^{\varphi_{2\ell}} \\
 \Delta  \ar[rr]^{\delta_\zeta}
 && \V\subset \Delta_{2\ell}      }
$$
Near the node $p^i_j$ of the central fiber $C$, the fiber of ${\cal X}_\zeta$ over $s$ is  the set of $(x^i_j,y^i_j,s)\in \cx^3$ satisfying $x^i_jy^i_j=\zeta^i_j s^{ n_j}$
and the pullback of (\ref{LocalMap2})  is a  map $f_\zeta:\X_\zeta\to \D$ which, by  (\ref{NumberOfchain}), is given locally by
\begin{equation}
\label{step4-map}
G_\zeta(x^i_j,y^i_j, s)\ =\
\big((x^i_j)^{m_j}, (y^i_j)^{m_j},  s^{|m|}\big).
\end{equation}

\bigskip
\non
{\bf Step 5 -- Blow-up:}\ \  The surface  $\X_\zeta$ is singular at the nodes $p^i_j$ when $n_j>1$.  The singularities can be resolved by repeatedly blowing up,    as follows.  Suppressing $i$ and $j$ from the notation, $\X_\zeta$  is locally given by $xy=\zeta s^{n_j}$ with   $C_1$ given by $y=0$ and $E_0=E_j$ given by $x=0$. \\[-2mm]

\noindent{\sc  First blowup:}  Blow up along  the locus $y=s=0$ by setting $y=y_1s$ and pass to the proper transform.   This introduces an exceptional curve $E_1$ on $C_0$ with coordinates $y_1$ and $x_1=1/y_1$.  The proper transform is given by

\begin{figure}[!h]
\center
{\small
$ \begin{cases}
xy_1= \zeta s^{n_j-1} \ \ \mbox{near $C_1\cap E_1$}\\
x_1y=  s \ \ \mbox{near $E_1\cap E_0$}.
\end{cases}
\hspace{2cm}
  \begin{tikzpicture}[start chain]
\LabeledCircle{}{$C_1$}
\LabeledCircle{}{$E_1$}
\LabeledCircle{}{$E_0$}
\end{tikzpicture}$
}\end{figure}

\noindent{\sc  Second blowup:}  Blow up along  $y_1=s=0$ by setting $y_1=y_2s$.   This introduces  $E_2$  with coordinates $y_2$ and $x_2=1/y_2$; the  proper transform is given by
\begin{figure}[!h]
\center
{\small
$ \begin{cases}
xy_2= \zeta s^{n_j-2} \ \ \mbox{near $C_1\cap E_2$}\\
x_2y_1=  s \ \ \mbox{near $E_2\cap E_1$}.
\end{cases}$
\hspace{2cm}
  \begin{tikzpicture}[start chain]
\LabeledCircle{}{$C_1$}
\LabeledCircle{}{$E_2$}
\LabeledCircle{}{$E_1$}
\LabeledCircle{}{$E_0$}
\end{tikzpicture}
}\end{figure}

\noindent Blowing up  $n_j-1$ times, and repeating on the other side of $E_0=E_j$  and at each node $p^i_j$, yields a smooth surface $\C_\zeta$ and a diagram
$$
\xymatrix{
\C_\zeta \ar[dr]  \ar[drrr]^{\F_\zeta} \ar[ddr]_{\varphi_\zeta} \\
&\X_\zeta  \ar[rr]_{G_\zeta}   \ar[d]^{\wh{\varphi}_\zeta}
&&  \D  \ar[d]^{\rho} \\
&\Delta  \ar[rr]^{s\to s^{|m|}}  
&& \Delta      }
$$
The central fiber of  $\C_\zeta \to \Delta$ is as described in Theorem~\ref{TwistingProp},  and all other fibers are smooth.  Using (\ref{step4-map}) and the equations $x=\zeta s^{n_j-n}x_n$ and $y=s^ny_n$,  one sees that, for $1\leq n< n_j$,
the map $\F_\zeta:\C_\zeta\to \D$ is given locally near $E_n\cap E_{n-1}= \{y_{n-1}=s=0\}\cap \{x_n=s=0\}$ by
\begin{equation}
\label{RamF1}
\F_\zeta(x_n,y_{n-1},s) \ =\
\big( (x_n)^{m_j(n_j-n+1)}\, (y_{n-1})^{m_j(n_j-n)}, \  (x_n)^{m_j(n-1)}\, (y_{n-1})^{nm_j}, \   s^{|m|}\big)
\end{equation}
with $x_n y_{n-1}=s$ where $y_0=y$, and  near $E_{n_j-1}\cap C_1$ by the same formula with   $C_1$ and $x$ labeled as $E_{n_j}$  and $x_{n_j}$ and with $xy_{n_j-1}=\zeta s$.

\vspace{5mm}

We can now relate the fibers of $\C_\zeta$ to the spaces $\CZ_{m,f,r}$ in (\ref{DecMr}). Note that for each vector $\zeta$ as in Step~4, the restriction of $\F_\zeta$ to the fiber over  $r=s^{|m|}\not= 0$ is a map
\begin{equation}\label{RestrictionMap}
f_{\zeta,s}\ =\ \F_\zeta|_{C_{\zeta,s}}:C_{\zeta,s}\to D_r.
\end{equation}

\begin{lemma}
\label{modulispacesequal}
Whenever $s\not= 0$ and $r=s^{|m|}$, we have
\begin{equation}\label{SM-Fmoduli}
\CZ_{m,f,r}\ =\ \underset{ \zeta\in Q_m}{\tcup}\ \{\,f_{\zeta,s}\}
\end{equation}
where the union is overall vectors $\zeta$.
\end{lemma}

\begin{proof}
Recall that $f:C\to D_0$ has contact marked points $q^i_j$ over $q^i\in D_0$ with multiplicities given by an odd partition $m^i$ for $1\leq i\leq k+3$.  By our choice of $q^i$ in Step~1, around each $q^i_j$ the map $\F_\zeta$ is
\begin{equation}\label{RamF2}
(x,s)\ \to\  (f(x),s^{|m|}).
\end{equation}
Hence the pull-back $\F_\zeta^*Q^i$ of $\D\to\Delta$
consists of
$\ell(m^i)$ sections $Q^i_j$ given by $Q^i_j(s)=(q^i_j,s)$.
After marking the points $Q^i_j\cap C_{\zeta,s}$, each of the $|m|$ maps   (\ref{RestrictionMap})
has contact marked points
$Q^i_j(s)$ over $Q^i(r)$ with multiplicity $m^i_j$, and thus lies  in the space
$\M_r$ of (\ref{def-Mr}).  As $r=s^{|m|}\to 0$ we have
 $f_{\zeta,s}\to f$ in the  Gromov topology; in particular,
the stabilization of the domain $C_{\zeta,s}$ converges to $C$.
The lemma follows because  $|Q_m|=|m|^2=|\CZ_{m,f,r}|$.
\end{proof}

\bigskip
\non
{\bf Step 6 -- Twisting at nodes:}\ \
The  pullback  $\F_\zeta^*\L$ of  the spin structure $(\L,\Phi)$ on the family $\D\to \Delta$ is not a
  theta characteristic on the fibers of $\C$.   In this step we  twist $\F_\zeta^*\L$ by a divisor $A$  to produce a line bundle
\begin{equation}
\label{defLZeta}
 \N_\zeta\ =\ \F_\zeta^*\L\otimes \O\big(\tfrac12 Q+ A\big).
\end{equation}
  over $\C_\zeta$ with the properties described in Theorem~\ref{TwistingProp}:  it  restricts to a theta characteristic on the general fiber, and  is especially simple on the central chains $\bar{E}_j$.  This twisting is crucial for later computations.

Specifically, let $Q=\sum (m^i_j-1)Q^i_j$ be the divisor on $C_\zeta$ as above and let $A=\sum A_j$ where
\begin{equation}
\label{defnofA}
A_j\ =\
\begin{cases}
\displaystyle
\tfrac{ n_jm_j-2}{2} \,E_j\ +\  \sum_{n=1}^{ n_j-1}\,\tfrac{(n_j-n)m_j-2}{2}\,  (E^1_{j;n}+E^2_{j;n})
  \qquad
\   &\mbox{if}\  m_j\  \mbox{is\ even,} \\[5mm]
\displaystyle
\tfrac{ n_j(m_j-1)}{2}\, E_j\ +\
\sum_{n=1}^{ n_j-1}\, \tfrac{(n_j-n)(m_j-1)}{2}\,  (E^1_{j;n}+E^2_{j;n})
&\mbox{if}\ m_j\ \mbox{is\ odd.}
\end{cases}
\end{equation}
To compute the restriction of $\N_\zeta$ to the fibers of $\C_\zeta$ we note a general fact:  fix  any irreducible component $\chi_m$  of $C_0$ and consider  the bundle $\O(\chi_m)$ on $\C$. For each other  component $\chi_n$, let  $P_{mn}$ be the divisor $\chi_m\cap \chi_n$.
By restricting  local defining functions one sees that   the  restriction of $\O(\chi_m)$ to a  general fiber $C_s$  and to $\chi_n$ are:
\
\begin{equation}
\label{4.HM}
 \O(\chi_m)\big|_{C_s} = \O, \quad\
 \O(\chi_m)\big|_{\chi_n} = \O\big(P_{mn}\big) \quad\mbox{for $m\not= n$},
\quad\
 \O(\chi_n)\big|_{\chi_n} =  \O\big(\textstyle{-\sum_{m\not= n} } P_{mn}\big) .
\end{equation}

\bigskip

\non
{\bf Proof of Theorem~\ref{TwistingProp}.}
For each $\zeta$ and $s\ne 0$,
the ramification divisor of the map
$f_{\zeta,s}$  in (\ref{RestrictionMap}) is $Q|_{C_{\zeta,s}}$,  and
by (\ref{4.HM}) the restriction of $\N_\zeta$ to $C_{\zeta,s}$ is
$$
N_{\zeta,s} \ =\  f_{\zeta,s}^*(\L|_{D_r})\otimes \O\big(\tfrac12 Q|_{C_{\zeta,s}}\big)
$$
Thus, as in (\ref{thetaonC}),  $N_{\zeta,s}$ is a theta characteristic on $C_{\zeta,s}$ and $f_{\zeta,s}$ has the associated parity $p(N_{\zeta,s})$.
Therefore (\ref{Summary-Shiffer}) follows from Lemma~\ref{modulispacesequal}.
This completes the proof of part(a) of Theorem~\ref{TwistingProp}.   Part(c) follows similarly, using (\ref{4.HM}) and noting that
 $f_i=\F_\zeta|_{C_i}$ has ramification index $m_j$ at the node in $C_i\cap E^i_{j;n_j-1}$.   Part(b) was shown in Step~5 above.  Finally, Part(d) follows by successively  applying (\ref{4.HM}), taking  $\chi_i$ to be the various $E^i_{j;n}$ and observing that $Q$ is disjoint from the chains $\bar{E}_{j}$ and that
 \begin{itemize}
\item
$\F_\zeta^*\L|_{E^i_{j;n}}=\O$ for $n=1,\cdots,n_j-1$ because the image $\F_\zeta(E^i_{j;n})$ is a point,
\item
$\F_\zeta^*\L|_{E_j}=\O(m_j)$ since  $\F_\zeta|_{E_j}=f|_{E_j}:E_j\to E$ has degree $m_j$ and $\L|_E=\O(1)$. \qed
\end{itemize}

\vspace{3mm}
\begin{center}
\psfrag{A}{$C_1$}
\psfrag{B}{$C_2$}
\psfrag{E1}{$E_1^0$}
\psfrag{E2}{$E_2^0$}
\psfrag{E3}{$E_3^0$}
\psfrag{E}{$E$}
\psfrag{D1}{$D_1$}
\psfrag{D2}{$D_2$}
\psfrag{F1}{$E^1_{1k}$}
\psfrag{F2}{$E^2_{1k }$}
\includegraphics[scale=1]{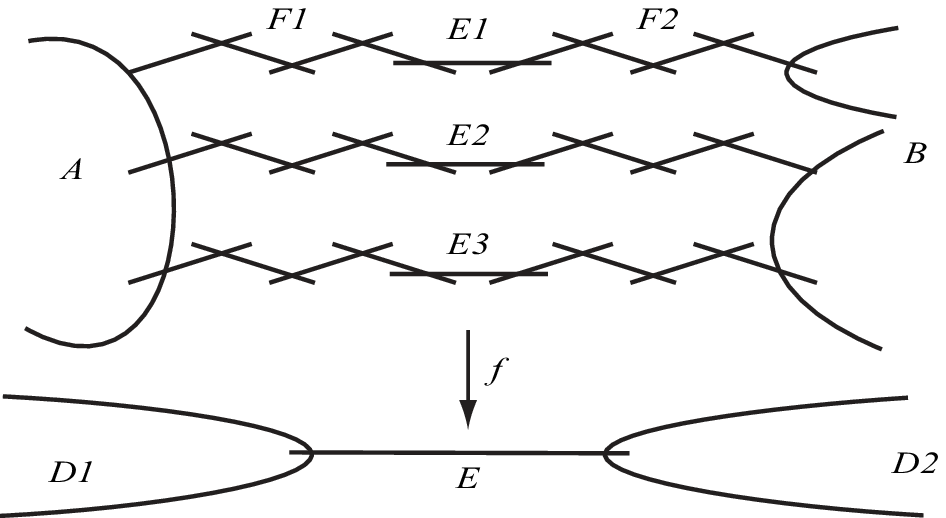}
\end{center}

\vspace{.6cm}


\setcounter{equation}{0}
\section{The operators $L_t$ on the family}
\label{section6}
\bigskip

For each $\zeta$, we now have an algebraic family $\C_\zeta \to\Delta$ and a bundle $\N_\zeta$ on $\C_\zeta$.    One can then apply the construction of Section~\ref{section2} to the fibers of $\C_\zeta$ to obtain operators
\begin{equation}
\label{5.1}
L_{s,t} = \del_{C_s} + tR_s: \Omega^{0}(C_{\zeta,s}, N_{\zeta,s}) \to \Omega^{0,1}(C_{\zeta,s},  N_{\zeta,s})
\end{equation}
that are a family version of the operators (\ref{Defdelt}).  This section describes a  global construction on the complex surface $\C$ whose restriction to fibers gives the operators (\ref{5.1}).   The global construction will be important  in later sections to obtain estimates on $L_{t,s} $ that are uniform in $s$.

 \medskip

\begin{lemma}
\label{firstpsilemma}
Each spin structure on $\D$ determines  a  nowhere-zero section $\psi$ of   $K_\C\otimes (\N^*_\zeta)^2\otimes \O(-\hat{A})$ where, with the same notation as (\ref{defnofA}),
\begin{equation*}
\hat{A}\ =\ \sum_{j=1}^\ell \hat{A}_j
\ \ \ \ \  \mbox{with}\ \ \ \ \
\hat{A}_j\ =\
\left\{
\begin{array}{ll}
\displaystyle
2 E_j \ +\
\sum_{n=1}^{ n_j-1}2(E^1_{j;n}+E^2_{j;n})   \qquad
\   &\mbox{if}\  m_j\  \mbox{is\ even,} \\[5mm]
\displaystyle
n_j E_j \ +\  \sum_{n=1}^{ n_j-1} (n_j-n)  (E^1_{j;n}+E^2_{j;n})
&\mbox{if}\ m_j\ \mbox{is\ odd}
\end{array} \right. \\[1mm]
\end{equation*}
\end{lemma}

\begin{proof}
The spin structure (\ref{homomorphism}) on $\D$ vanishes to first order along $E\subset D_0$, so defines a section $\phi$ of $K_\D\otimes (\L^*)^2\otimes \O(-E)$.
Noting that $\O(D_0)=\O$, we can write $\O(-E)$ as $\O(D_1+D_2)$.  Using the definition (\ref{defLZeta}), the pullback $\psi=\F^*_\zeta\phi$ is then a section of
\begin{equation}
\label{psipullbackhome}
\F_\zeta^*(K_\D)\otimes \O(Q)\otimes \O(2A)\otimes (\N_\zeta^*)^2\otimes \F_\zeta^*\O(D_1+D_2).
\end{equation}
Recall that the ramification divisor ${\cal R}_{\F_\zeta}$ of the map $\F_\zeta$
has local defining functions given by the Jacobian of $\F_\zeta$. One can thus see from (\ref{RamF1}) and (\ref{RamF2})  that
${\cal R}_{\F_\zeta} =Q + |m|C_{\zeta,0}$. Choosing a trivialization $\O(C_{\zeta,0})=\O$, the Hurwitz formula  gives
\begin{equation}\label{HurFor-S}
\ K_{\C_\zeta}\ =\ \F^*K_\D\otimes \O({\cal R}_{\F_\zeta})\ =\  \F_\zeta^*K_\D\otimes \O(Q).
\end{equation}
From the second equation in (\ref{RamF1}) we also have
\begin{equation}\label{pull-back-E}
 \F_\zeta^*\O(D_1+D_2)\  =\
\O\Big(\,|m|C_1+|m|C_2 +\sum_{j=1}^\ell \sum_{n=1}^{n_j-1}nm_j(E^1_{j;n}+E^2_{j;n})\,\Big)
\end{equation}
because $\{v^i=0\}\subset D_i$ and $\{y^i_{n-1}=0\}\subset E^i_{j;n}$.
Together with the fact $\O(|m|C_{\zeta,0})=\O$,
the last two displayed equations  imply that
the right-hand side of (\ref{psipullbackhome}) is $K_\C\otimes (\N^*_\zeta)^2\otimes \O(-\hat{A})$.
\end{proof}

\medskip

\begin{cor}
\label{cor5.2}
There  is a conjugate-linear bundle map
$R_\zeta: \N_\zeta  \to \bar{K}_\C \otimes \N_\zeta$ whose divisor is $\hat{A}$.
\end{cor}

\begin{proof}
Choose a global section $a$ of $\O(\hat{A})$ with divisor $\hat{A}$.  Then with $\psi$ as in Lemma~\ref{firstpsilemma},  $\psi\otimes a$ is a section of  $K_\C\otimes (\N^*)^2$ whose divisor is $\hat{A}$. Regarding this as a map $\hat{\psi}:  \N_\zeta \ \to\ K_\C\otimes \N_\zeta^*$ and composing with the (conjugate-linear) star operator $\bar{*}: \Omega^{2,0}(\C,  \N_\zeta^*)\to  \Omega^{0,2}(\C,  \N_\zeta)$ gives a bundle map
\begin{equation}
\label{5.psihat}
R_\zeta =\bar{*}\,\hat{\psi}:\, \N_\zeta \ \to\ \bar{K}_\C \otimes \N_\zeta
\end{equation}
 with divisor $\hat{A}$.
 \end{proof}

Because  $\C$ is a smooth surface, the canonical bundle $K_\C$ is isomorphic to the relative dualizing sheaf  $\w_\zeta$ of $\varphi_\zeta:\C\to \Delta$.  In fact, the restrictions of $K_\C$ and $\w_\zeta$ are related by the commutative diagram
\begin{equation}
\label{5.diagram}
\xymatrix{
\w_\zeta\otimes \N^*_\zeta |_{C_s} \ar[rr]_{\bar{*}_s}  \ar[d]_{\wedge ds}
 &&  \bar{\w}_\zeta\otimes \N_\zeta |_{C_s}   \ar[d]^{\wedge d\bar{s}} \\
\ \ \  K_\C\otimes \N^*_\zeta |_{C_s} \  \ar[rr]^{\bar{*}}
 &&\ \ \ \bar{K}_\C\otimes \N_\zeta |_{C_s} \      }
\end{equation}
where $\bar{*}$ is as in Corollary~\ref{cor5.2},    $\bar{*}_s$ is the similar operator on the fiber $C_s$ of $\C$, and all four arrows are isomorphisms.   In local coordinates $(x,y,s)$ near a node $xy=s$ of $C_s$,   we have $ds=x dy+ydx$ and   $\w_\zeta$ is freely  generated by $\tau=\frac{dx}{x}=-\frac{dy}{y}$.  The star operator  on $C_s$ is multiplication by $i$ on $(1,0)$ forms and by $-i$ on $(0,1)$-forms, so $\bar{*}\tau =  *\bar{\tau}=-i\bar{\tau}$.  The diagram commutes because, after restricting to $C_s$ and suppressing the bundle coordinates,   $\tau\wedge ds = \frac{dx}{x}\wedge(xdy+ydx) =  dx\wedge dy$ and hence
$$
\bar{*}(\tau\wedge ds) = *(d\bar{x}\wedge d\bar{y}) = -i *(d\bar{x}\wedge d\bar{y}) =-i *(\bar{\tau}\wedge d\bar{s}) = (\bar{*}_s\tau)\wedge d\bar{s}.
$$

Diagram (\ref{5.diagram}) implies that  for each $s$ there is a section $\psi_s$ of $\w_\zeta\otimes \N_\zeta^*$ on $C_{\zeta,s}$ such that $\psi_s\wedge ds$ is the section $\hat{\psi}$ in (\ref{5.psihat}).   Consequently, for each $s$, $R_s=\bar{*}_s\psi_s$ is a conjugate-linear bundle map
\begin{equation}\label{HomR}
R_s : N_{\zeta,s} \ \to\ \bar{\w}_\zeta \otimes N_{\zeta,s}
\end{equation}
between bundles on the curve $C_{\zeta,s}$. Let $N_{\zeta,i}=\N_\zeta|_{C_i}$ for $i=1,2$.

\begin{theorem}
\label{theorem6.3}
The map (\ref{HomR})  satisfies Properties (\ref{2.threeproperties}).   Furthermore,
\begin{itemize}
\item[(a)]  On each smooth fiber   ${C_{\zeta,s}}$, $R_s$ is an isomorphism $N_{\zeta,s} \to \bar{K}_{C_{\zeta,s}} \otimes N_{\zeta,s}$.
\item[(b)]  For $i=1,2$, the restriction of $R_0$  to $C_i$ is a map  $R_i:N_{\zeta,i} \to  \bar{K}_{C_i} \otimes N_{\zeta,i}$
with divisor ${\cal R}_{f_i}^{ev}$.
\end{itemize}
\end{theorem}

\begin{proof}
 The proof of Lemma~\ref{Rlemma} shows that $R_s$ satisfies  Properties (\ref{2.threeproperties}). By Diagram (\ref{5.diagram})  we have $R_s\wedge d\bar{s}= \bar{*}_s\psi_s\wedge d\bar{s} = R_\zeta$, so the divisor of $R_s$ is $\hat{A}\cap C_{\zeta,s}$.  Statement (a) holds because this intersection is empty for $s\not=0$.    For (b), note that the restriction of $\w_\zeta$ to $C_i$ is $K_{C_i}\otimes \O(\sum_j p^i_j)$, so the divisor of $R_s$ is $C_i\cap\hat{A} -\sum_j p^i_j = {\cal R}_{f_i}^{ev}$.

\end{proof}

It is useful to have a local  formula for $R$ around the nodes $p^i_j$ where $C_i$ meets the chain $\bar{E}_j$.   As in (\ref{RamF1}), we have local coordinates $(x,y,s)$ around $p^i_j$ in which $C_1=\{y=s=0\}$ and $E^i_{j,n_j-1}=\{x=s=0\}$.  By Corollary~\ref{cor5.2} and the definition of $\hat{A}$, there is a local nowhere-zero section $\nu$ of $\N_\zeta$ and a constant $a\in\cx^*$ such that   $R(\nu) = a \bar{x}^p\, \bar{\tau}\otimes\nu$ where   $p=2$ if $m_j$ is even and $p=1$ is $m_j$ is odd.  By replacing $\nu$ by $e^{i\theta}\nu$, we can assume that $a$ is real and positive.  Thus after writing $\tau$ as $dx/x$ we have
\begin{equation}
\label{5.Requation}
 R(\nu)\Big|_{\bar{E}_j}\ =\ 0 \hspace{2cm}
 R(\nu)\Big|_{C_i}\ =\
 \begin{cases}
a \bar{x}\ \  d\bar{x}\otimes \nu   & \mbox{$m_j$  even} \\
a \ \   d\bar{x}\otimes \nu     \qquad &   \mbox{$m_j$ odd.}
\end{cases}
\end{equation}
 for some real $a>0$.  Similarly, one finds that at each interior nodes of $\bar{E}_j$, there are local coordinates in which $R(\nu) = a \bar{x}\, \bar{y}^2\   d\bar{x}\otimes \nu$.

\medskip

We conclude this section by stating  two facts about the index of the operators (\ref{5.1}).

\begin{lemma}
\label{indexlemma}
For $s\not= 0$,  the operator $L_{s,t}$ on $C_s$ has index 0, and for $i=1,2$  $\ind L_{0, t}\big|_{C_i}=-\ell^{ev}$
where $\ell^{ev}$ is the number of even ramification points of $f_i=f_0|_{C_i}$.
\end{lemma}

\begin{proof}
For each $s$, $L_{s,t}$ is a compact perturbation of the $\del$-operator, so its index is  twice of the holomorphic Euler characteristic $\chi(N_{\zeta,s})$. 
But $\chi(N_{\zeta,s})=0$ for $s\not= 0$ because $N_{\zeta,s}$  is  a theta characteristic on  $C_s$.    Similarly, for $i=1,2$, $\N|_{D_i}$ is a  theta characteristic on $D_i$ so $2\deg(\N|_{D_i})=2h-2$.  Theorem~\ref{TwistingProp}c,  the
 Riemann-Roch and  Riemann-Hurwitz formulas then give
\begin{equation}
\label{5.index}
2\chi(N_{\zeta,i})\
=\ -\deg(f^*_iTD_i) +\deg({\cal R}_{f_i}-{\cal R}^{ev}_{f_i}) +  \chi(C_i)\ =\ - \deg({\cal R}_{f_i}^{ev})\ =\ - \ell^{ev}.
\end{equation}
\end{proof}

\vspace{.6cm}


\setcounter{equation}{0}
\section{Bundles of Eigenspaces}
\label{section7}
\bigskip

In Section~\ref{section5} we constructed curves  $\C_\zeta\to\Delta$   over the disk  whose general fibers  are smooth and whose central fiber $C_0$ is a   union  $C_1\cup  E \cup C_2$ of nodal curves  where $C_1$ and $C_2$ are disjoint and
$$
\bar{E}=\cup_j \bar{E}_j
$$
where each $\bar{E}_j$ is the chain of rational curves (\ref{Adiagram}).  For simplicity, we will drop $\zeta$ from our notation.  There is also  is a line bundle $\N\to \C$ whose restriction $N_s$  to each  fiber $C_s$ comes with the bundle map $R_s$ described in Theorem~\ref{theorem6.3} and the one-parameter family of operators
$$
L_t\  =\  \del+tR_s
$$
 To take adjoints, we fix  a hermitian metric on $\N$ and a Riemannian metric $g$ on $\C$, with $g$ chosen to be  Euclidean in the local coordinates $(x,y,s)$ around in node of $C_0$ (as described in Section~\ref{section5}).

On each curve $C_s$, the operator $L_t^*L_t$ on $N_s$ has non-negative real  eigenvalues $\{\la\}$ that vary continuously with $s$ for $s\not=0$.     Given a function $\la_1(s) >0$ on $\Delta$  (we will fix a value later),   consider the family of    vector spaces $\E\to \Delta$ whose fiber over $s$ is spanned by the  low eigensections as in (\ref{3.defEF}):
\begin{equation}
\label{Elowspace}
\E_s\ =\ \mbox{span}_\R\, \left\{ \xi\in L^2(C_s; N_s) \ \big| \   \mbox{$L_t^*L_t\xi=\la\xi$ for $\la<\la_1$}         \right\}.
\end{equation}
The eigensections of  $L_tL_t^*$ give a similar family  $\F\to\Delta$ of $L^2$ sections:
\begin{equation}
\label{Flowspace}
\F_s\ =\ \mbox{span}_\R\, \left\{ \eta\in L^2(C_s; \overline{K}_{C_s}\otimes N_s) \ \big| \   \mbox{$L_tL_t^*\eta=\la\eta$ for $\la<\la_1$}         \right\}
\end{equation}
and $L_t$ is a bounded finite-dimensional linear map $L_t: \E_s\to \F_s $.  In general, the  dimension of such eigenspaces can jump as $s$  varys.   This section  establishes  conditions under which  $\E$ and $\F$  are actually vector bundles over $\Delta$.

\medskip

We will show that the spaces of $\E_s$ can be modeled on the space of holomorphic sections of $\N$ along the central fiber $C_0$.

 \begin{lemma}
 \label{WjLemma}
 Let $\E_0 =\{\mbox{\rm continuous }\psi\in H^0(C_0, N_0)\}$. There are  $L^2$ orthogonal decompositions
 \begin{equation}
\label{EFlowhigh}
 \E_0 = W \oplus \E_0' \hspace{1.6cm}    W =   \bigoplus_{ j | m_j \, even}\hspace{-2mm}  W_j 
 \end{equation}
 where $W=\ker L_t\cap\E_0$, each $W_j$  is a 1-dimensional complex space  and  $\E' _0 \cong  H^0(C_1, N_1) \oplus H^0(C_2, N_2)$.  Furthermore, $\F_0=\ker L_{0,t}|_{C_0}$   has real dimension $2\ell^{ev}$.
 \end{lemma}

 \begin{proof}
Because $R$ is non-trivial on $C_1\cup C_2$  and trivial on $\bar{E}$, the  proof of Theorem~\ref{2.vanishingtheorem}  shows that any continuous $\psi\in\ker L_t$  vanishes on $C_1\cup C_2$ and is holomorphic on $\bar{E}$,  so lies in the   direct sum  of the $L^2$ orthogonal   complex vector spaces
\begin{equation}
\label{defWj}
 W_j=\left\{\mbox{continuous  $\psi\in H^0(C_0,N_0)$ with support on $\bar{E}_j$}  \right\}.
\end{equation}
If $m_j$ is odd, $N_0$ is $\O(1)$ on the center component  of $\bar{E}_j$ and is trivial the other irreducible components;  the boundary conditions (\ref{defWj}) then imply that $W_j=0$.  If $m_j$ is even, $N_0$ is $\O(1)$ on the first and last  components of $\bar{E}_j$ and trivial on the others;  hence $W_j\cong \cx$ and each $\psi\in W_j$ is constant on   $\bar{E}_j$ except on the end components.  

One similarly sees that each $\psi\in H= H^0(C_1, N_1) \oplus H^0(C_2, N_2)$ extends  continuously  and holomorphically   over  $C_0$;  the extension is  unique  modulo $W$ and hence there is a unique extension $\bar{\psi}$  perpendicular to $W$.  Let  $\E_0' \cong H$ denote the set of all extensions.  Then for each continuous $\xi\in H^0(C_0, N_0)$ there is a $\bar{\psi}\in \E_0'$ so that $\xi-\bar{\psi}$ has support in $\bar{E}$, and therefore lies in $W$ as above.  Thus $\E_0$ decomposes as in   (\ref{EFlowhigh}).

 Finally,  note that the restriction of each $\eta\in \F_0=\ker L_{0,t}^*$ to each component of $\bar{E}$ satisfies $(\del^*+tR^*)\eta=0$ with $R=0$, so  by Theorem~\ref{TwistingProp}d lies in $H^{01}(\P^1,\O)=0$ or $H^{01}(\P^1,\O(1))=0$.  Thus $\eta=\eta_1+\eta_2$ where $\eta_i$ lies in the kernel of the operator $L_i=L_{0,t}|_{C_i}$.  But Theorem~\ref{2.vanishingtheorem} and Lemma~\ref{indexlemma}  show that
$$
\dim \ker L^*_i\ =\ \dim \ker L_i-\ind L_i\ =\ 0- (-\ell^{ev})\ =\ \ell^{ev}
$$
so we conclude that $\F_0$ has real dimension $2\ell^{ev}$.
 \end{proof}

\medskip

The following theorem shows that the  decomposition Lemma~\ref{WjLemma}  on the nodal curve $C_0$ carries over to  nearby smooth curves.  Parts (a)  and (b) cover the case where $|t|$ is small,  part (d) covers the case where $|t|$ is large, and (c) holds for all $t$.  The upshot is that the low eigenspaces are of three types:  one whose eigenvalues grow linearly with $t$, 
one whose eigenvalues are logarithmically small in $|s|$, and one whose eigenvalues are bounded by $|s|^2(1+t^2)$ and which splits  as a  sum of 2-dimensional eigenspaces. 
\begin{equation}
\label{7.defl(s)}
\la_1(s) = \frac{c_0}{|\log |s||}
\end{equation}

\begin{theorem}
\label{Elowtrivial}
(a) There is a  $c_0>0$ such that,  with $\la_1(s)$ as in (\ref{7.defl(s)}) and  $0<|s|, |t| \ll 1$,  the  low eigenspaces (\ref{Elowspace}) and (\ref{Flowspace})   form  vector bundles $\E_W,  \E'$ and $\F'$ over $\Delta$ and $\F^0$ over $\Delta\ssetminus \{ 0\}$    and  a  diagram of bundle maps
\begin{align} 
\label{kervsE2}
 \begin{CD}
\Delta \times \left(W \oplus   \E_0'\right)   @>\Phi>\cong>  \, \E_W  \oplus \E' \\
 @.     @VV{L_t} V  \\
 @.\F^0 \oplus\F'
 \end{CD}
\end{align} 

\noindent (b) There are positive constants $C_1, C_2,  C_3$ such that for $t\not= 0$ 
\begin{equation} 
\label{7.Evaluebounds}
\E_W  =  \bigoplus \left\{ E_\la\, \big|\, \la\leq C_1|s|^2(1+ t^2) \right\} \qquad  \E' = \bigoplus \left\{ E_\la\,\big|\, C_2t^2 \leq \la\leq C_3 (|s|^2+t^2) \right\}.
\end{equation}

\noindent (c)  For   $t\not= 0$ and $|s|\ll 1+t^2$,  the first component of $\Phi$ is a bundle isomorphism 
\begin{equation}
\label{kervsE2b}
 \begin{CD}
\Delta \times  \bigoplus_j  V_j   @>\Phi^V>\cong> \bigoplus_j \, \E_j   \\
 \end{CD}
\end{equation}
where the $\E_j$ are  real rank 2  bundles that are $L^2$ orthogonal up to terms of order $O(|s|\sqrt{1+t^2})$. \vspace{1mm}

\noindent
(d)  For each $\tau>0$ there is a $\delta>0$ such that     (\ref{kervsE2b}) is an isomorphism onto the sum of the eigenspaces with eigenvalue $\la \leq C_1|s|^2(1+t^2)$ whenever  $|t|\geq \tau$ and  $|s|<\delta$.  
 \end{theorem}

The proof of Theorem~\ref{Elowtrivial} occupies the rest of this section.  The method is straightforward:  transfer elements of $\ker L_t$  on  $C_0$ to $C_s$ by extending and cutting off, then estimate using   the coordinates introduced in Section~\ref{section5}.

\begin{proof}
    For each node $p$ of $C_0$, the construction of Section~\ref{section5} provides  coordinates $(x,y)$ on a ball $B(p,\ep)\subset\C$  in which  $C_s =\{xy=\zeta s\}$.  After shrinking $\ep$ we may assume these balls are disjoint and that on each ball there is  a local holomorphic section $\nu$ of $\N$ with $\frac12\leq|\nu|^2\leq 2$ pointwise.  Let
$B(\ep)$ be the union of these balls.  Each $\psi\in W\oplus \E_0'$ is continuous and can be extended as follows:
\begin{itemize}
\item  On  $C_0\cap B(\ep)$,  $\psi$ has the form $f\nu$ for some continuous holomorphic function. Extend this to the section $\psi^{in}=F\nu$ by setting
$$
F(x,y)=f(x,0)+f(0,y)-f(0,0)
$$
on each $B(p,\ep)$.  This extension is continuous, holomorphic and agrees with $\psi$ along $C_0$.
\item The construction of Section~\ref{section5}  shows that  $C_0\ssetminus B(\ep)$ is a disjoint union of embedded smooth curves.  Hence we can  extend $\psi$ to a smooth section  $\psi^{out}$ of $\N$ on a neighborhood of $C_0\ssetminus B(\ep)$ by parallel translation in the normal directions;  the normal component of  $\nabla\psi^{out}$ then  vanishes along $C_0$.

\end{itemize}
To merge  the above extensions, fix a  smooth bump function $\beta_\ep$  supported on $B(2\ep)$ with $\beta_\ep=1$ on $B(\ep)$ and with  $0\leq \beta_\ep\leq 1$ and $|d\beta_\ep|\leq 2/\ep$ everywhere.  Then
\begin{equation}
\label{6.Psibeta}
\widehat{\psi}\ =\ \beta_\ep \psi^{in} + (1-\beta_\ep) \psi^{out}
\end{equation}
is a smooth extension of $\psi$ to a section of $\N$  on a neighborhood of $C_0$.  After choosing an $L^2$ orthonormal basis $\{\psi_k\}$ of $W\oplus\E'_0$, this construction creates extensions 
$\{\widehat{\psi}_k\}$.    We can then define a linear map $\Psi_s:W\oplus \E_0\to C^\infty(C_s,N_s)$ for each small $s$ by setting
\begin{equation}
\label{6.PsiHat}
\Psi_s(\psi_k)\ =\ \psi_{k,s} \quad \mbox{where $\psi_{k, s}= \widehat{\psi}_k\big|_{C_s}$}
\end{equation}
for each basis vector $\psi_k$ and extending  linearly. For each $j$, $\psi_s=\psi_{k,s}$ is continuous,   holomorphic on $C_s\cap B(\ep)$, and satisfies the following    bounds for $|s|<1$:
\begin{enumerate}
\item[(i)]  Because  $\psi^{in}$ and $\psi^{out}$ are continuous extensions of  $\psi$, we have $ |\psi^{in}_s -\psi_s^{out}| \leq c_1(\ep)|s|$  on  the  region $A_s(\ep) =C_s\cap(B(2\ep)\ssetminus B(\ep))$, which contains the support of $d\beta_\ep$.  
\item[(ii)]   On the complement of $B(\ep)$, the  curves $C_s$ converge to $C_0$ in $C^1$ as $s\to 0$ and $\del\psi_{k,0}=0$.  Hence
$|\del\psi_{k,s}^{out}|\leq c_2(\ep) |s|$  on  the support of $1-\beta_\ep$.
\end{enumerate}
The  $L^2$ norm of $L_t\psi_s =   \del\beta_\ep (\psi^{in}_s -\psi_s^{out}) + (1-\beta_\ep)\, \del\psi^{out}_s +tR\psi_{k,s}$ therefore  satisfies
\begin{equation}
\label{6.sbound}
 \|L_t \psi_{k,s}\|^2
\ \leq\  c_3|s|^2\left(  \int_{A_s(\ep)}  \frac{8}{\ep^2}\ +\   \mbox{Area}(C_s)\right)\ + c_4 t^2  \|\psi_{k,s}\|^2
 \ \leq\ c_5\, (|s|^2+t^2)\, \|\psi_{k,s}\|^2
\end{equation}
 where the last inequality holds because  $R$ is bounded and  $\|\psi_{k,s}\|\to\|\psi_k\|=1$  as $s\to 0$. 

 If $\psi_k\in W$ then  (\ref{6.sbound}) can be strengthened.  There is a basis $\{\psi_j\}$ of $W$ where the
 support of  $\psi_j$ lies in an even chain $\bar{E}_j$ and $R=0$  along that chain;   we therefore have $|R\psi_{j,s}|\leq c_6|s| |\psi_{j,s}|$ outside the $2\ep$-balls around the even nodes $p_j^i$.  In those $2\ep$-balls, there are local coordinates $(x, y)$  in which $xy=\zeta s$ on $C_s$ and $R$ has the form (\ref{5.Requation}) and $\psi_k =by$ for some $b\in\cx$ (cf. Theorem~\ref{TwistingProp}d).  Therefore $|R \psi_{j,s}| \leq c_3 |\bar{x}\bar{y}| = c_7 |s|$ and   (\ref{6.sbound}) becomes
\begin{equation}
\label{6.sboundW}
 \|L_t \psi_{j,s}\|^2
\ \leq\ c_8\, |s|^2(1+ t^2)\, \|\psi_{j,s}\|^2.
\end{equation}
The constant $c_5$ and $c_8$ can be taken independent of $j$ and $k$, and hence (\ref{6.sbound}) holds for all   $\psi\in \E_0$ and  (\ref{6.sboundW}) holds for all   $\psi\in W$.
 
 We also have a lower bound on $\|R\psi_s\|$ for $\psi_s\in\E'$.   In this case, $\psi$ is holomorphic and is non-zero on an open set in $C_1\cup C_2$.  The   facts that $|R|$ is non-zero almost everywhere on $C_i$   and $\|\psi_s\|\to\|\psi_k\|=1$  as $s\to 0$ imply that, for small $|s|$, 
  \begin{equation}
\label{6.lowerbound}
 \|R \psi_s\|^2
 \ \geq \ \int_{C_i\ssetminus B(2\ep)}  t^2 |R|^2\, |\psi_s|^2
 \ \geq\ c_9t^2 \int_{C_i\ssetminus B(2\ep)}   |\psi_s|^2
 \ \geq\ c_{10}t^2  \|\psi_s\|^2.
\end{equation}

\medskip

   At this point we can define $\E$ and the decomposition $\E=\E_W\oplus \E'$ by projecting onto  low eigenspaces.  For this we assume that $s$ is not zero and is small enough that $|s|< c_5(|s|^2+t^2) <\frac12  \la_1(s)$ with $\la_1(s)$ as  in (\ref{7.defl(s)}). Applying Lemma~\ref{LAlemma} below  twice  shows that:
   \begin{itemize}
   
   \item The composition $\Phi_s= \pi_{s} \Psi_s: \E_0 \to \E_s$ of $\Psi_s$ with the $L^2$ orthogonal projection into  the sum of the eigenspaces $E_\la$ on $C_s$  with  $\la\leq c_5(|s|^2+t^2)$  is an isometry up to terms of order  $O(|s|+|t|)\, )$.  
      
\item The composition $\Phi_s^W= \pi_{s} \Psi_s^W: W \to \E_s$ of $\Psi_s^W$ with the $L^2$ orthogonal projection into  the sum $\E_W$ of the eigenspaces $E_\la$ on $C_s$  with  $\la\leq c_8 |s|^2(1+t^2)$ is an isometry up to terms of order $O(|s|\sqrt{1+t^2}\, )$; it  has the form $\pi_W\Phi_s$ for small $|s|$ and $|t|\leq T$.  
\end{itemize}
 Because basis elements $\{\psi_j\}$ of $W=\oplus W_j$ have  disjoint support, the image   $\Phi^W(\oplus W_j)$ defines real rank 2 subbundles $\E_j\subset \E_W$   as in (\ref{kervsE2b}).

 Now let  $\E'$ be the orthogonal complement of $\E_W$ in $\E$. Each eigenvector  $\psi\in\E'$ with eigenvalue $\la$  and norm 1 can be written as an orthogonal sum $\psi_{s}+v$ with $\psi_s$ in the image of (\ref{6.PsiHat}) and  $v\in\E_W$ satisfying $\|v\|\leq c_8(\sqrt{\ell(s)} +|t|)\, \|\psi_s\|$.    We then obtain a lower bound on $\la=\|L_t\psi\|^2$  using  (\ref{2.vanishintegral}),  the inequality $2(a+b+c)\geq a^2-4b^2-4c^2$ and (\ref{6.lowerbound}), noting that $R$ is bounded and $\psi$ has unit norm:
$$
\la\ \geq\  t^2\|R\psi\|^2
\ \geq\ \frac{t^2}{2}\left[ \|R\psi_s\|^2 - 4 \|Rw\|^2 - 4\|Rv\|^2\right]
\ \geq\ \frac{t^2}{4}\left[ c_{11} -c_{12}\left(\ell(s) +|t|^2\right)\right].
$$
For small $|s|$ and $|t|$, this gives the inequality $\la\geq C_2t^2$ in (\ref{7.Evaluebounds}).

\medskip

 In fact, one can choose the constant $c_0$ in the definition (\ref{7.defl(s)}) of  $\la_1$ so that  $\Phi_s:\E_0\to\E_s$   is surjective.   The proof, which is crucial but rather technical, is given in the appendix.

\medskip

To finish, set $\F^0=L_t(\E_W)$ and $\F'=L_t(\E')$ and observe that  $L_t$ maps the non-zero eigenspaces of $L_t^*L_t$ isomorphically to the eigenspaces of $L_tL^*_t$ with the same  eigenvalues. But $\ker L_t = 0$ for $s\not = 0$ by Theorem~\ref{theorem6.3} and $\ker L_t=W$ on $C_0$ by Lemma~\ref{WjLemma}, so after shrinking $\Delta$, $\F'$ is a bundle over $\Delta$ and $\F^0$ is a bundle over $\Delta\ssetminus\{0\}$.  Finally, given $\tau>0$, we have $C_1|s|^2(1+\tau^2)<\min\{\la(s), C_2\tau^2\}$ for all small $|s|$;  the eigenvalue bounds (\ref{7.Evaluebounds})  then  show that  the sum of the eigenspaces in Theorem~\ref{Elowtrivial}d is exactly $\E_W$.

 \end{proof}

The proof of Theorem~\ref{Elowtrivial}  made use of the following  elementary  lemma.

\begin{lemma}
\label{LAlemma}
Let $L:H\to H'$ be a bounded linear map between Hilbert spaces so  that   all eigenvalues of $L^*L$ lie in $[0, \mu]\cup [\la_1, \infty)$ with $0<\mu < \la_1$.  Consider the  low eigenspace
$$
E_{low} = \bigoplus_{\la\leq \mu}  E_\la
$$
 and suppose that  $V\subset H$ is a  subspace with $|Lv|^2\leq c\mu\, |v|^2$ for all $v\in V$.  Then the orthogonal projection $\pi:V\to E_{low}$ is the identity plus an operator  of order $O(\sqrt{\mu})$.
\end{lemma}

\begin{proof}
 Fix $v\in V$ and write $v=v_0+w$ where $v_0=\pi v$ and $\langle v_0, w\rangle =0$.  Then $\langle Lv_0, Lw\rangle = \langle L^*Lv_0, w\rangle$ vanishes because $L^*Lv_0\in E_{low}$, while $|Lw|^2 \geq \la_1\, |w|^2$ because $w\perp E_{low}$.  Thus
 $\la_1\, |w|^2 \leq |Lw|^2  = |Lv|^2 -  |Lv_0|^2  \leq  c_1\mu |v|^2$, which means  that $|v-\pi v|=|w|\leq c_2 \sqrt{\mu}\,  |v|$.
\end{proof}

\vspace{.6cm}


\setcounter{equation}{0}
\section{Parity formulas}
\label{section8}
\bigskip

As Section~\ref{section7}, we fix a partition $m$, a map $f=(f_1,f_e,f_2)$ in $\M_{m,0}$ and $\zeta\in Q_m$; these data determine maps $f_{\zeta,s}:C_{\zeta, s}\to D_s$.  
 Theorem~\ref{TwistingProp}  shows that  for $s\not= 0$ the restriction of $\N$ is a theta characteristic $N_s$ on $C_s$, so defines a  parity $p(f_{\zeta,s})$.   In fact,  by Theorem~\ref{Injective-Esp},   $p(f_{\zeta, }s)$ is  the TR spectral flow of the  finite-dimensional  linear map
$$
L_{s,t}\  =\  \del+tR_s:  \E_s\to \F_s 
$$
 between the fibers of the bundle of Theorem~\ref{Elowtrivial}.  Moreover, this sign is independent of $s\not= 0$ and $t \not=  0$.   In this section we will express the parity as a product of $2\times 2$  determinants.  

\medskip

  When the partition $m$ is odd,   $f_1$ and $f_2$  themselves have   parities given by
the theta characteristics $N_1$ and $N_2$ on $C_1$ and $C_2$ (cf.  Theorem~\ref{TwistingProp}), and these  determine the parity of $f_{\zeta,s}$.

  \begin{lemma}
\label{oddLemma}
If $m$ is odd  then for every $\zeta\in Q_m$and  $s\ne 0$  the  parity of $f_{\zeta,s}$    is 
$$
p(f_{\zeta,s}) \ =\ p(f_1)\cdot p(f_2).
$$
\end{lemma}

\begin{proof}
 If $m$ is odd, Lemma~\ref{WjLemma} shows that $W=0$  and   the complex dimension of $\E_0$ is $h^0(N_1)+ h^0(N_2)$.  By  the discussion in Section~\ref{section3}, $p(f_{\zeta,s})$ is $\sgn \det L_{s,t}:\E'_s\to \F'_s$, and this is  independent of $s$  for small $|s|$ and $|t|$ in the trivialization of  Theorem~\ref{Elowtrivial}a.  But for $s=0$,   $L_{0, t}=t R_0|_{\E_0}$ is a complex anti-linear isomorphism and therefore, as in (\ref{3.SFformula2}),  
 $$
 \sgn \det L_{0,t}\ =\ (-1)^{h^0(N_1)+ h^0(N_2)}\ =\ p(f_1)\cdot p(f_2).
 $$
\end{proof}

 \bigskip

  If $m$ is not an odd partition, the parity can be partially computed by the method of Lemma~\ref{oddLemma}.

\begin{theorem}
\label{evenThm}
For each partition $m$ and  $s\ne 0$, and for every $\zeta\in Q_m$  and $t\not= 0$,  the  parity of $f_{\zeta,s}$    is given by
 \begin{equation}
 \label{parityformula-even2}
p(f_{\zeta,s})\ =\ (-1)^{h^0(N_1)+ h^0(N_2)}    \prod_{j\,|\, m_j\,  even}\hspace{-2mm}  \sgn  \det  L_t  \Big|_{\E_j}.
 \end{equation}
\end{theorem}

\begin{proof}
Theorem~\ref{Injective-Esp}  again shows that the parity is $\sgn \det L_{s,t}$ where $L_{s, t}$ is the map $L_t$   in Theorem~\ref{Elowtrivial} on the fiber over $s\not= 0$.  Since $L_t$ preserves eigenspaces and $\ker L_t=0$ for non-zero $s$ and $t$, we have
$$
p(f_{\zeta,s})\ =\  \sgn \det L_{s,t}\big|_{\E'_s}\cdot \sgn\det L_{s,t}\big|_{\E_W}.
$$
The first factor is equal to $p(f_1) p(f_2)$ as in the proof of Lemma~\ref{WjLemma}.  To decompose the second factor,  choose an $L^2$ orthonormal basis of $\E_W$  consisting of eigenvectors  $\psi^i_j\in\E_j$  of $L_{s,t}^*L_{s,t}$ with eigenvalues $\la^i_j$.   Then $\|L_{s,t}\psi^i_j\|^2=\la^i_j$, while Theorem~\ref{Elowtrivial}c gives 
$$
|\langle L\psi^i_j, \, L\psi^i_{j'} \rangle|\ =\ 
|\langle L^*L\psi^i_j, \, \psi^i_{j'} \rangle|\ =\ 
 \la^i_j \,|\langle  \psi^i_j, \, \psi^{i'}_{j'}\rangle|
 \ \leq\ c_1|s| \sqrt{1+t^2}\   \la^i_j
 $$
 whenever $j'\not= j$.   Thus for fixed $t$ and $0<|s|\ll t$, the matrix of $L_{s,t}$ on $\E_W$  has a  block form whose off-diagonal entries that are arbitrarily   small compared to the diagonal entries, giving (\ref{parityformula-even2}).
\end{proof}

We conclude this section by observing that (\ref{parityformula-even2}) remains valid when $L_t$ is replaced by a perturbation of the form  $\hat{L}_t =  L_t+\ep tS$ for certain $S$.  Specifically, applying  Theorem~\ref{2.vanishingtheorem}  and  the inequality $2t|(\del\xi, S\xi) |\leq |\del\xi|^2 +t^2|S\xi|^2$, we have
\begin{equation}
\label{7.lastdisplay}
\int_{C_{\zeta,s}} |\hat{L}_t\xi|^2\ =\ \int_{C_s} |L_t\xi|^2 + 2t\ep (\del\xi, S\xi) + \ep^2t^2 |S\xi|^2
\ \geq  \int_{C_{\zeta,s}} (1-\ep)|\del\xi|^2 +t^2\left( |R\xi|^2-\ep |S\xi|^2\right).
\end{equation}
Now recall from (\ref{5.Requation}) that $R$ has the local expansion
$R(\nu)= a\bar{x}d\bar{x}\,\nu$ at each even node $p=p^i_j$. Take $S$ of the same form:  $S(\nu)= b\bar{x}d\bar{x}\,\nu$ near $p$ and bumped down to 0 outside a small neighborhood of $p$.   Then there are constants $c_1, c_2$ such that
 $$
 |S\xi|^2\leq c_1 r^2 |\xi|^2 \leq c_2 |R\xi|^2
 $$
 Substituting into (\ref{7.lastdisplay}) shows that  there is an $\ep_0$ such that  $\ker \hat{L}_t=0$ for all   $\ep\leq \ep_0$.  This means that $\sgn \det \hat{L}_t =\sgn \det L_t$,  so Proposition~7.3 holds with $R$ replaced by
$$
(R+\ep S)(\nu) = (1+\ep b)\, \bar{x}d\bar{x}\,\nu +\dots
$$
 for small $\ep$.  In this sense we are free to replace the leading coefficient   in the Taylor expansion of $R$ by any small perturbation and still have formula (\ref{parityformula-even2}).

\vspace{.6cm}


\setcounter{equation}{0}
\section{Concentrating eigensections}
\label{section9}
\bigskip

The last factor in the parity formula (\ref{parityformula-even2}) is independent of non-zero $s$ and $t$.  In this and the next section we explicitly evaluate (\ref{parityformula-even2})  by first taking $t$ large, and then $s$ small.
The key observation is  that as $t\to\infty$ the elements of $\ker L^*_t$ on $C_0$ concentrate around the points where $R$ vanishes, and that on nearby smooth curves $C_s$ the low eigensections of $L_t^*L_t$  similarly concentrate with essentially explicit formulas.

  \medskip

  On each smooth curve $C_s$, the adjoint of $L_t$ is the map $L_t^*:\Omega^{0,1}(N_s)\to \Omega^0(N_s)$ given by
\begin{equation}
\label{defR*}
  L_t^*\ =\ \del^* + tR^*
\end{equation}
where $R^*$ (the pointwise adjoint of $R$)  is a real bundle map that satisfies $R^* J =-J R^*$.  Thus $R^*$ is zero at those points where $R=0$, and is an isomorphism at all other points of $C_s$.

\begin{lemma}
\label{Aendolemma}
$A= \del R^*+R\del^*$ is a bundle endomorphism and for each $s\not= 0$
\begin{equation}
\label{L*LIntegral}
\int_{C_s} |L_t^*\eta|^2\ =\ \int_{C_s} |\del_t^*\eta|^2 + t\langle\eta, A\eta\rangle + t^2\, |R^*\eta|^2 \qquad \forall \eta\in\Omega^{0,1}(C_s, N_s).
\end{equation}
\end{lemma}

\begin{proof}
Formula (\ref{L*LIntegral}) follows immediately  from (\ref{defR*}).
Clearly $A$ is a first order linear differential operator, so is a bundle endomorphism if its symbol is 0.  For a non-zero tangent vector $v$, the symbols $\sigma_v$ of $\del$ and $-\sigma^*_v$ of $\del^*$ are isomorphisms, in fact, $\sigma_v\sigma^*_v=|v|^2\,\mbox{Id.}$ Taking the symbol of   equation    (\ref{2.threeproperties}c)  gives  $R^*\sigma_v =\sigma_v^* R$.    But then $-|v|^2$ times  the symbol is $A$ is
$$
-|v|^2\, (\sigma_v R^*-R\sigma_v^*)\ =\  \sigma_v R^* \sigma_v\sigma_v^* - \sigma_v\sigma_v^* R \sigma_v^*
\ =\ \sigma_v\left[R^*\sigma_v-\sigma_v^* R\right] \sigma_v^*\ =\ 0.
$$
\end{proof}

\begin{lemma}
\label{concentrationthm}
For each  neighborhood ${\mathcal B}$ of the set of zeros of $R^*$ there is a constant $c>0$ such that for all $t\geq 1$ each  solution of $L^*_sL_t\eta=\la \eta$ with $\la\leq 1$ satisfies
$$
\int_{C\setminus {\mathcal B}} |\eta|^2\ \leq\ \frac{c}{t}\, \int_C |\eta|^2.
$$
\end{lemma}

\begin{proof}
Noting that $R^*$ is an isomorphism  on $C\setminus{\mathcal B}$ and applying (\ref{L*LIntegral}) gives the inequalities
$$
\int_{C\setminus {\mathcal B}} |\eta|^2\ \leq\
\frac{c}{t^2}\,\int_{C\setminus {\mathcal B}} t^2 |R^*\eta|^2\ \leq\
\frac{c}{t^2}\int_{C} |L_t^*\eta|^2+t  \left|\langle\eta, A\eta\rangle\right| \ \leq\ \left(\frac{c\la}{t^2}+\frac{c}{t} \|A\|_\infty \right)\int_C|\eta|^2.
$$
\end{proof}

Lemma~\ref{concentrationthm} means that as $t\to \infty$ the low eigensections of $L_t^*L_t$  concentrate in small neighborhoods $D(\ep)$ of the zeros of $R^*$.   The zeros occur only at the nodes with even multiplicity, where $R$ is given by (\ref{5.Requation}).  In particular, the elements of $\ker L^*_t$ on $C_0$ concentrate at the even nodes $p^i_j$;  these are explicitly described in the next lemma.

Writing $\eta=\phi\, d\bar{x}\otimes\nu$ in the coordinates of (\ref{5.Requation}),  the equation $L_t^*\eta=0$ takes the form
\begin{equation}
\label{localR*}
-\frac{\partial\phi}{dx}+ at\,\bar{x}\bar{\phi}\ =\ 0
\end{equation}
with $a>0$. Regarded as an equation on $\cx$, this  has the explicit   $L^2$-normalized solution
\begin{equation}
\label{defphi}
\eta=\phi\, d\bar{x}\otimes  \nu  \qquad  \mbox{where }\ \    \phi(x) \ =\ i\,\sqrt{\tfrac{at}{\pi}}\  e^{-atx\bar{x}}. 
 \end{equation}
 By cutting off and gluing,  these forms give approximate elements of $\ker L^*_t$ on curves.   For example, we can glue onto $C_1$ as follows.   Fix disjoint disks $D_j=D(p^1_j, 2\ep)$ in $C_1$  with coordinate $x$ centered on the points $p^1_j$ of even multiplicity.   Choose a  cutoff function $\beta_j=\beta_\ep$  on $D_j$  as defined before (\ref{6.Psibeta}) and set
\begin{equation}
\label{DefFApproxeta}
\F^{approx}_{t}\ =\ \mbox{span}_\R\, \big\{\eta_j=\beta_j \cdot \phi(x)\, d\bar{x}\otimes \nu \, \big|\, j=1, \dots, \ell^{ev}\big\}.
\end{equation}

\begin{lemma}
\label{concentrationC1}
For large $t$, the $L^2$ orthogonal projection $\pi_a: \F^{approx}_{0,t}\to \ker L^*_t$  on  $C_1$ is  an  isomorphism and   an isometry up to terms of order $O(1/t)$.
\end{lemma}

\begin{proof}
  Integration in polar coordinates shows that $\frac12\leq\|\eta_j\|\leq 2$ for all $j$ and all large $t$.
Also,  $L^*_t\eta_j=(\del +tR^*)(\beta_j\eta) = \beta_j L^*_t\eta -*(\del\beta\wedge *\eta)$ with $L_t^*\eta=0$.  Integrating using (\ref{defphi}) yields
\begin{equation}
\label{8.3est1}
\|L^*_t\eta_j\|^2\ \leq\ \int_{D_j} |d\beta|^2\, |\eta|^2\ \leq\ \frac{c_1}{\ep^2}\int_\ep^{2\ep} \phi^2(r)\ rdr
\  \leq\ \frac{c_2}{t^2}\, \|\eta_j\|^2
\end{equation}
after  noting that  $t^2e^{-2at\ep^2} \leq \ep^2$ for  large $t$.  Lemma~\ref{LAlemma} then shows that $\pi_a$ is an isometry up to terms of  order $1/t$.   It is an isomorphism  because
 the $\{\eta_j\}$  are linearly independent (they  have disjoint support) and $\ker L^*_t $  and $\F^{approx}_{0,t}$ have the same dimension $\ell^{ev}$  by Lemma~\ref{WjLemma}.
\end{proof}

Lemma~\ref{concentrationC1} is easily modified to apply to the  smooth fibers $C_s$ of $\C\to\Delta$.   For each node $p^i_j$ of $C_0$ with even multiplicity, let $\beta^i_j$ to be the function $\beta_\ep$ as in (\ref{6.Psibeta}) in $(x,y)$ coordinates on the ball $B(p^i_j,2\ep)$ in $\C$ and replace (\ref{DefFApproxeta}) by the $2\ell^{ev}$-dimensional real vector space
$$
\F^{approx}_{t}\ =\ \mbox{span}_\R\, \big\{\eta^i_j=\beta^i_j\cdot  \phi(x)\, d\bar{x}\otimes \nu \, \big|\, j=1, \dots, \ell^{ev}, i=1,2\big\}.
$$
The restriction to $C_s$ followed by the $L^2$ orthogonal projection gives a linear map $\pi_a: \F^{approx}_{t}\to \F^{low}_{t}$ onto the low eigenspace of $L_tL^*_t$.

\begin{theorem}
\label{concentrationC2}
Whenever $0<|s|\leq 1/t^2$ and $t$ is   large,  $\pi_a: \F^{approx}_{t}\to \F^{low}_{t}$    is an isomorphism and an  isometry up to terms of order $O(1/t)$.
\end{theorem}

\begin{proof}
For each $i,j$, the support of $\eta_j^i$ lies in the portion of $C_s$ given by $(x, \zeta s/x)$ for $|s|/2\ep\leq |x|\leq 2\ep$ with metric (\ref{6.metric}).    Integration in polar coordinates  shows that   $\frac12\leq\|\eta^i_j\|\leq 2$ for all large $t$.    Noting that the support of $d\beta_\ep$ lies in $A\cup A'$ where $A=\{ \ep\leq r\leq 2\ep\}$ and $A'=\left\{|s|\leq 2\ep r\leq 2|s|\right\}$.   Then the $L^2$ norm of $L^*_t\eta$ is bounded  by the first  integral in (\ref{8.3est1}) with the domain  $D_j$  replaced by $A\cup A'$.   On $A$, the  metric (\ref{6.metric}) approaches the eucidean metric as $s\to 0$, so the bound  (\ref{8.3est1}) holds.  On $A'$,  we can replace the conformally invariant quantity  $|d\bar{x}|^2\ dvol_s$  by its value in  the eucidean metric, namely $2 r dr d\theta$ and   replace $|d\beta_\ep|^2$ by  its euclidean value times $\gamma^{-1}$.   Noting that $|d\beta_\ep|^2 \gamma^{-1} \leq4 |\ep s|^{-2}\left(1+|s|^2 r^{-4}\right)^{-1}\leq c_1\ep^{-2}$ on $A'$ we have,  as in (\ref{8.3est1}), 
\begin{equation}
\label{e4xineqality}
 \int_{A'}  |d\beta_\ep|^2 \, |\eta^i_j|^2\ dvol_s
 \ \leq\ \frac{c_2}{\ep^2}\int_{\frac{|s|}{2\ep}}^{\frac{|s|}{\ep}} e^{-2atr^2}\ rdr
\ \leq\  c_3 \frac{t |s|^2}{\ep^4}
\ \leq\ \frac{c_4}{t^2}
\end{equation}
where we have used the inequalities $|s|\leq 1/t^2$ and  $e^{-x}-e^{-4x}\leq 4x$ for  small $x$ and assumed that $t\geq \ep^{-4}$.   Combining these bounds yields
 \begin{equation}
\label{concentrationC2estimate}
\|L^*_t\eta^i_j\|^2\ \leq \   \frac{c_5}{t^2} \, \|\eta^i_j\|^2.
\end{equation}
Lemma~\ref{LAlemma} then shows that $\pi_a$ is an isometry up to $O(1/t)$ terms.  It is an isomorphism because (\ref{kervsE2}) implies that  for $s\not= 0$ $\F^{low}_{t}\cong \E_W\cong W$ has real dimension $2\ell^{ev}$.
\end{proof}

\vspace{.6cm}


\setcounter{equation}{0}
\section{Cancellation for even partitions}
\label{section10}
\bigskip

For each partition $m$ and each $\zeta\in Q_m$,     Theorem~\ref{evenThm} expresses the parity $p(f_{\zeta,s})$ in terms of the linear operators $L_{t,j}$ between the low eigenspaces $\E_j^{low}$ and $ \F_{t,j}$ described in  Theorem~\ref{Elowtrivial}   and, for large $t$,   Theorem~\ref{concentrationC2}.
In this section we will use the concentration principle of Section~\ref{section9}  to show the following remarkable cancellation property.

 \begin{theorem}
 \label{cancellationThm}
Let $m$ be an even partition as above and  $s\ne 0$. Then 
$$
\sum_{\zeta\in Q_m}  p(f_{\zeta,s})\ =\ 0.
$$
\end{theorem}

To prove Theorem~\ref{cancellationThm},  fix an even partition $m=(m_1,\cdots,m_\ell)$ and  $\zeta=(\zeta_1,\zeta_1',\cdots,\zeta_\ell,\zeta_\ell')$ in $Q_m$  and choose an even  component $m_j$ of $m$.
 We will focus on the  chain $\bar{E}_j$ corresponding to the chosen $m_j$ and the nodal points 
 $p=p^1_j\in C_1\cap \bar{E}_j$  and  $q=p^2_j\in C_2\cap \bar{E}_j$ at the two ends of $\bar{E}_j$.
For any  bases
$\{\psi_1, i\psi_1\} \ \mbox{of }  \E_j^{low} $ and $\{\eta_1, \eta_2\} \ \mbox{of }  \F_j^{low}$
the $j^{th}$ factor in (\ref{parityformula-even2})  is   the sign of the determinant of the matrix
\begin{equation}
\label{Lmatrix}
 L_{t,j}\ =\ L_t  \big|_{\E_{j}^{low}} \ =\ \begin{pmatrix} ( \eta_1, L_t\psi_1) &  ( \eta_2, L_t\psi_1) \\[1mm]
 ( \eta_1, L_t\psi_2) &  ( \eta_2, L_t\psi_2)
\end{pmatrix}
\end{equation}
whose entries are given by  conformally invariant  $L^2$ inner products
$$
(\eta, \xi )\ =\ \int_{C_{\zeta,s}}  \mbox{Re}(\eta\wedge *\overline{\xi} )   \hspace{1cm} \eta, \xi\in\Omega^{0,1}(C_s, N_s)
$$
on smooth fibers $C_{\zeta,s}$ of $\C_\zeta$.   Theorems~\ref{Elowtrivial}   and \ref{concentrationC2}  give explicit formulas for sections $\psi_j$ and $\eta_k$ which give bases up to terms of order $O(\sqrt{|s|})$; using these in (\ref{Lmatrix}) will correctly give $\sgn \det  L_{t,j}$ for small $s$.

The results of Section~\ref{section9} show that for large $t$  the inner products in the first column of (\ref{Lmatrix})  are concentrated near $p$, and those in the second column are concentrated near $q$.   Thus $\det L_{t,j}$ can be regarded as the contribution of an ``instanton''  tunneling across the chain $\bar{E}_j$ between $p$ and $q$.

To proceed, we need coordinate formulas for $\psi, \eta_1$ and $\eta_2$.   Recall that there   are local coordinates $(x,y)$ and a local holomorphic section $\nu$ of $\N$ defined  a ball $B(p^1_j,2\ep)$ so that $C_{\zeta,s}$ is locally given by $xy=\zeta s$, $|\nu(p)|=1$,  and 
$$
R(\nu)\,=\, a \bar{x}\, d\bar{x} \otimes \nu
$$
 for a positive real constant $a$  (cf.  (\ref{5.Requation})).    Noting that   elements in $W_j$ vanish to order 1 at $p$ and $q$, we can take  $\psi_1$  and $\eta_1$ to be  the restrictions of 
\begin{equation}
\label{9.psietaeq}
\psi\ =\ \beta(r)\, b y\, \nu   \hspace{1.5cm} 
\eta \ =\ \frac{i}{2\pi} \,  \beta(\rho)\,  e^{-atr^2}\, d\bar{x} \otimes \nu   \hspace{.7cm}
\end{equation}
 to  $C_{\zeta,s}$ where $b\in\cx^*$,     $r=|x|$, $\rho^2=|x|^2+|y|^2$ as described in  (\ref{6.PhiHat}) and  (\ref{defphi}) but with $\eta$ normalized so that its $L^2$ norm satisfies $\|\eta\|^2\approx (4\pi a t)^{-1}$ for large $t$.

\begin{lemma}
\label{innerproductMu}
There is a $T$   such that whenever $t>T$ and $0<|s|\leq  1/t$ we have
\begin{equation}
\label{9.2}
  (\eta, L_t\psi)_{C_{\zeta, s}}  \ =\    a\Re(ibs\zeta)  \  e^{-at |s|^2/4\ep^2} \  +\  O\left(\tfrac{1}{\sqrt{t}}\right).
\end{equation}
\end{lemma}

\begin{proof}
Writing $L_t\psi=\del\psi+tR\psi$ with $R\psi = \beta \bar{b}\bar{y} R(\nu) = \beta \bar{b}  a\,\overline{xy}\, d\bar{x}\otimes \nu$ and using  the equation $xy=\zeta s$, one sees that the $L^2$ inner product is  
$$
  (\eta, L_t\psi)_{C_{\zeta,s}}\
\ =\  I\ +\ \frac{at}{2\pi} \Re(i bs\zeta) \int_{C_{\zeta, s}}  \beta(\rho)\beta(r) e^{-atr^2} \ |d\bar{x}|^2\, |\nu|^2\ dvol_s 
$$
with $|I|\leq \|\eta\|\cdot\|\del\psi\|\leq c_1 |s|/t\leq c_1/\sqrt{t}$ by (\ref{6.sbound}),  our normalization of $\eta$ and the hypothesis on $s$. As in the proof of Theorem~\ref{concentrationC2},  we can replace $|d\bar{x}|^2\ dvol_s$ by  $2 r dr d\theta$.   Writing $|\nu|^2=1+h_1$  with  $|h_1|\leq c_2r$ and integrating over $\theta$ gives
$$
 (\eta, L_t\psi)_{C_{\zeta, s}} \ \ =\  2at \,  \Re\left(i bs\zeta\right)\,  \int_{|s|/2\ep}^\infty (1+ (\beta-1) +h_2)\, e^{-atr^2}\ r dr\   +\  O\left(\tfrac{1}{\sqrt{t}}\right).
$$
where $\beta=\beta(\rho)\beta(r)$ satisfies  $|\beta-1|\leq 1$ and  $|h_2|\leq c_3 r$.  The first and the last parts of this integral can be estimated using the formulas
$$
 \int_{|s|/\ep}^\infty   e^{-at r^2}\ r dr \ =\ \tfrac{1}{2at}  e^{-at|s|^2/4\ep^2}
\hspace{1in}
  \int_0^\infty r^2  e^{-at r^2} dr \ =\ \tfrac{\sqrt{\pi}}{4}\, (at)^{-3/2}.
  $$
Noting  that $\beta-1=0$ for $|s|/\ep \leq r\leq \ep$ and estimating as in (\ref{e4xineqality}), the middle integral is dominated by
$$
\int_{|s|/2\ep}^{|s|/\ep} e^{-atr^2}\ r dr +\int_\ep^\infty e^{-atr^2}\ r dr
\ \leq\  \frac{-1}{2at}\left[ e^{-atr^2}\Big|_{|s|/2\ep}^{|s|/\ep}\ +\  e^{-at\ep^2}\right]
\ \leq\ c_4\left(|s|^2+\frac{1}{t^2}\right).
$$
The lemma follows.
\end{proof}

The remaining entries in (\ref{Lmatrix}) can be calculated from (\ref{9.2}).   Setting $\psi_1=\psi$,  $\psi_2=i\psi$ and $\eta_1=\eta$,  the substitution $b\mapsto ib$ gives
$$
( \eta_1, L_t\psi_2)_{C_{\zeta, s}}\ \ =\    -a\Re(ibs\zeta)\  e^{-at |s|^2/4\ep^2} \  +\  O\left(\tfrac{1}{\sqrt{t}}\right).
$$
The entries in the second column of (\ref{Lmatrix}) are evaluated using similar coordinates $(x_2,y_2,\nu_2)$ around $q$;  in these coordinates $R(\nu_2)=a_2\bar{x_2}d\bar{x_2}\otimes \nu_2$ for some real number $a_2>0$,  and  $\psi_1$ and $\eta_2$ have the form (\ref{9.psietaeq}) with $b$ replaced by a different constant, which we write as  $ib_2\in\cx^*$.    After a little algebra, one obtains
$$
 \det {L}_{t,j}\ =\   -a a_2
\left|
\begin{array}{ll}
\Re(i bs\zeta_j) & \Re(b_2s\zeta^\prime_j) \\
\Re( bs\zeta_j) &   \Re(ib_2 s\zeta^\prime_j)
\end{array}
\right|\ =\ a a_2\, |s|^2\, \left(\Re(\,b\bar{b_2}\,\zeta_j\overline{\zeta_j^\prime}\,)\ +\ O\left(\tfrac{1}{\sqrt{t}}\right)\right).
$$

\bigskip

\begin{proof}[Proof of Theorem~\ref{cancellationThm}]   By the remark at the end of Section~\ref{section8}   we may assume that $\Re(b \bar{b}_2\,\zeta_j\overline{\zeta_j^\prime})$ is non-zero for each $j$ with $m_j$ even.  For these $j$,  the above formula gives
 $\sgn \det L_{t,j} = \sgn  \Re(b\bar{b_2}\,\zeta_j\overline{\zeta_j^\prime})
$ when $t$ is large  and $0<|s|\leq 1/t$.    For  each $\zeta\in Q_m$, Theorem~\ref{evenThm} therefore shows that
\begin{equation}
\label{8.product}
 p(f_{\zeta,s})\ =\  (-1)^{h^0(N_1)+ h^0(N_2)}   \cdot  \prod \sgn  \Re(b\bar{b_2}\,\zeta_j\overline{\zeta_j^\prime})
\end{equation}
 where the product is over all $j$ with $m_j$ even.  

\medskip

Now  comes the punch line.   Fix an index $j$ with even $m_j$.   For each   $\zeta=(\zeta_1,\zeta_1',\cdots,\zeta_\ell,\zeta_\ell')$ in $Q_m$,
 replacing $\zeta_j$ by $-\zeta_j$ defines an involution  $\iota: Q_m\to Q_m$ that reverses the sign of  (\ref{8.product}).  Thus  the sum
 $$
 \sum_{\zeta\in Q_m}  p(f_{\zeta,s})\ =\
 \tfrac12 \sum_{\zeta\in Q_m}\Big[ p(f_{\zeta,s})+ p(f_{\iota(\zeta),s})\Big]\ =\ 0.
 $$
\end{proof}


Theorem~\ref{cancellationThm}   completes the proof of  Theorem~\ref{Main} --- the main result stated in the introduction.  Specifically,  Lemmas~\ref{modulispacesequal} and \ref{oddLemma} imply Theorem~\ref{MainTask}a,   Theorem~\ref{cancellationThm}  and (\ref{Summary-Shiffer}) imply Theorem~\ref{MainTask}b,  and the arguments at the end of Section~\ref{section4} showed how Theorem~\ref{Main} follows from Theorem~\ref{MainTask}.

\vspace{.6cm}


\setcounter{equation}{0}
\section{Calculational examples}
\label{section11}
\bigskip

This last section uses  Theorem~\ref{Main} to explicitly compute the degree $d=4$ spin Hurwitz numbers for every genus.  For degrees 1 and 2 the computation is trivial: since  the only odd partitions of 1 and 2 are (1) and $(1^2)$, by (\ref{trivial})
the degree $d=1,2$ spin Hurwitz numbers are the etale spin Hurwitz numbers
$$
H_1^{h,p}\ =\ (-1)^p, \qquad  H_2^{h,p}\ =\ (-1)^p \, 2^h,
$$
which are the GW invariants of K\"{a}hler surfaces calculated in \cite{LP1} and \cite{KL}.
For notational simplicity, we will write the spin Hurwitz numbers $H_{m,\cdots,m}^{h,p}$ with the same $k$  partitions $m$ of $d$
simply as $H_{m^k}^{h,p}$ and the etale spin Hurwitz number $H_d^{h,p}$ as  $H^{h,p}_{m^0}$.
The numbers 3 and 4 each have two odd partitions, namely $(3)$ and $(1^3)$, and  $(31)$ and $(1^4)$. Thus, by (\ref{trivial}), it suffices to compute
$H_{(3)^k}^{h,p}$ and $H_{(31)^k}^{h,p}$ for all $k\geq 0$.
The degree $d=3$ case is calculated in \cite{L}:
$$
H_{(3)^k}^{h,\pm}\ =\ 3^{2h-2}\big[\,(-1)^k2^{k+h- 1}\pm 1\,\big]
$$
where  $+$ and $-$ denote the even and odd parities.
Here we will compute the corresponding degree 4 invariants.

\begin{theorem}\label{computation} The degree 4   Hurwitz numbers are
$$
H_{(31)^k}^{h,\pm }\ =\ (3!)^{2h-2}\cdot 2^k \big[\pm 2^{k+h-1}+ (-1)^{k}\,\big] \qquad \mbox{for $k\geq 0$}.
$$
\end{theorem}

\medskip

We begin by computing three special cases.

\begin{lemma}
\label{lemma10.2}
(a) $H_4^{1,-}= 0$, \ (b)  $H_{(31)}^{1,-}= -6$ \ and \ (c)\
$H_{(31)^3}^{0,+}=\frac23$.
\end{lemma}

\begin{proof} For a genus one spin curve  with odd parity,
formula (3.12) of \cite{EOP} shows that
\begin{equation}\label{EOP}
H_{(31)^k}^{1,-}  \ =\  2^{-k} \left[\big(\f_{(3)}(31)\big)^k - \big(\f_{(3)}(4)\big)^k \right].
\end{equation}
Here  the so-called {\em central character} $\f_{(3)}$ can be written as
$\f_{(3)}= \frac13\,\p_3 + a_2\p_1^2 + a_1\p_1 + a_0$
for some $a_i\in{\mathbb Q}$ and  $\p_1$ and $\p_3$ are  the functions of partitions $m=(m_1,\cdots,m_\ell)$ of $d$
defined by
\begin{equation*}\label{ssf}
\p_1(m) \  =\  d - \tfrac{1}{24}
\ \ \ \ \ \ \ \mbox{and}\ \ \ \ \ \ \
\p_3(m) \ =\  \textstyle{\sum_j}\  m_j^3 \, -\,  \tfrac{1}{240}
\end{equation*}
 The case $k=0$ gives (a), and the case $k=1$ gives (b).

Next consider a map $f$ in the dimension zero relative moduli space $\M^V_{\chi,(31),(31),(31)}({\mathbb P}^1,4)$.
By the dimension formula (\ref{DimRelModuli}), $\chi=2$ and hence
the domain of $f$ is either a  rational curve  or  a disjoint union
of a rational curve  $C_0$ and an elliptic curve $C_1$.
Maps of the first type have parity $p(f)=1$ since $N_f=\O(-1)$.
For maps of the second type,
\begin{itemize}
\item
$f_0=f|_{C_0}\in\M^V_{2,(1),(1),(1)}({\mathbb P}^1,1)$  and  $N_{f_0}=\O(-1)$,
\item
$f_1=f|_{C_1}\in  \M^V_{0,(3),(3),(3)}({\mathbb P}^1,3)$ and  $N_{f_1}=\O$ (cf. the proof of Lemma~7.2\,b of \cite{L}).
\end{itemize}
It follows that $p(f)=p(f_0)\cdot p(f_1)= 1\cdot (-1)=-1$.  Thus by (\ref{introplainHurwitz}) and (\ref{SHN})  the difference  between the ordinary and spin Hurwitz numbers is twice the contribution of the maps of the second type:
$$
H_{(31)^3}^{0,+}\ =\ H_{(31)^3}^{0} - 2 H_{(1)^3}^0\cdot H_{(3)^3}^{0}.
$$
The three (ordinary) Hurwitz numbers on the right-hand side
 can be calculated by using  formula (0.10) of \cite{OP}.  This yields (c).
\end{proof}

\medskip

\begin{lemma}
\label{lemma10.3}
Theorem~\ref{computation} holds for genus $h=0$ and genus $h=1$.
\end{lemma}

\begin{proof}
Taking  $h=h_1=1$ and $p=p_1=1$ in  Theorem~\ref{Main}a   and using Lemma~\ref{lemma10.2}  gives
\begin{equation}\label{EOP2}
H_{(31)^2}^{1,-} \ =\  3\, H_{(31)}^{1,-}\cdot H_{(31)^3}^{0,+} \ =\ -12.
\end{equation}
Using (\ref{EOP2})  and  Lemma~\ref{lemma10.2}b  to evaluate  the $k=1$ and $k=2$ cases of (\ref{EOP}),  one sees that
$\f_{(3)}(31)\ =\ -4$ and $\f_{(3)}(4)\ =\ 8$.  Formula  (\ref{EOP}) then becomes
\begin{equation}\label{InitialTinv}
H_{(31)^k}^{1,-}\ =\ (-1)^k2^k - 4^k \qquad\mbox{for $k\geq 0$}.
\end{equation}
For $k\geq 1$,  we can apply Theorem~\ref{Main}a with $(h_1,p_1)=(1,-)$,  $(h_2,p_2)=(0,+)$ and $k_0=0$ and use Lemma~\ref{lemma10.2}a to obtain
$$
H_{(31)^{k-1}}^{1,-} \ =\
3\,H_{(31)}^{1,-}\cdot H_{(31)^k}^{0,+}\ =\
-3\cdot 3!\,H_{(31)^k}^{0,+}.
$$
Together with (\ref{InitialTinv}), this equation yields
\begin{equation}\label{Tinv-h=0-k2}
H_{(31)^k}^{0,+} \ =\  -\tfrac{1}{3\cdot 3!} \big[\,(-1)^{k-1}2^{k-1}-4^{k-1} \,\big]
\qquad\mbox{for $k\geq 1$,}
\end{equation}
and the same formula  holds for $k=0$ because the invariant  $H_{(31)^0}^{0,+}= H_4^{0,+}$ is $\tfrac{1}{4!}$.  Finally, combining (\ref{Tinv-h=0-k2}) with
 the formula of Theorem~\ref{Main}b with $(h,p)=(1,+)$,  shows that
\begin{equation}\label{Tinv-h=1}
H^{1,+}_{(31)^k} \ =\  3\,H_{(31)^{k+2}}^{0,+} + 4!\,H_{(31)^k}^{0,+}
\ =\  (-1)^k2^k+4^k.
\end{equation}

\end{proof}

\non
{\em Proof of Theorem~\ref{computation}:}
By Lemma~\ref{lemma10.3} we can assume that $h\geq 2$. Applying the formula of Theorem~\ref{Main}a with $(h_2,p_2)=(1,+)$, we obtain
\begin{equation*}
H_{(31)^k}^{h,p} \ =\
4!\,H_{(31)^0}^{h-1,p}\cdot H_{(31)^k}^{1,+} +
3\,H_{(31)}^{h-1,p}\cdot H_{(31)^{k+1}}^{1,+}.
\end{equation*}
From this, we can deduce the matrix equation
\begin{equation*}\label{deg4-final}
\left(
\begin{array}{l}
H_{(31)^k}^{h,p} \\ H_{(31)^{k+1}}^{h,p}
\end{array}
\right)
=
\left(
\begin{array}{ll}
4!\,H_{(31)^k}^{1,+} & 3\,H_{(31)^{k+1}}^{1,+} \\
4!\,H_{(31)^{k+1}}^{1,+} & 3\,H_{(31)^{k+2}}^{1,+}
\end{array}
\right)
\left(
\begin{array}{ll}
4!\,H_{(31)^0}^{1,+} & 3\,H_{(31)}^{1,+} \\
4!\,H_{(31)}^{1,+} & 3\,H_{(31)^2}^{1,+}
\end{array}
\right)^{h-2}
\left(
\begin{array}{l}
H_{(31)^0}^{1,p} \\ H_{(31)}^{1,p}
\end{array}
\right)
\end{equation*}
Theorem~\ref{computation} follows after inserting the values given by (\ref{InitialTinv}) and (\ref{Tinv-h=1}).
\qed

\vspace{9mm}

\appendix
\setcounter{equation}{0}
\renewcommand{\theequation}{A.\arabic{equation}}
\renewcommand{\thetheorem}{A.\arabic{theorem}}
\section{Appendix}

This appendix establishes the subjectivity statement needed  in the proof of Theorem~\ref{Elowtrivial}.    Let  $\E$ (resp. $E_W$)
be the image of the map $\Phi_s$ (resp. $\Phi_s^W$) defined below (\ref{6.lowerbound}).

 \begin{lemma}
\label{LemmaA3}  
Given $0<T$, there are constants $c_0, \delta>0$  such that whenever $|s|$ is sufficiently small all eigenspaces $E_\la$ with $\la |\log |s| | < c_0$ satisfy
\begin{equation}
\label{lemmaA1eq}
(a)\ \ E_\la\subset \E \mbox{ for $|t|\leq \delta$}
\hspace{3cm}
(b)\ \ E_\la\subset \E_W \mbox{  for $T<|t|$.}
\end{equation}
\end{lemma} 

\begin{proof}
Otherwise there would be  sequences $t_n\to \tau$ and $s_n\to 0$ and $L^2$ normalized eigensections   $\xi_n$ on $C_n=C_{s_n}$ with eigenvalues satisfying $\la_n |\log |s_n|| \to 0$ and with 
   $L^2$ orthogonal to $\E$ on $C_n$ with $t_0=0$ in case (a),   and  $L^2$ orthogonal to $\E_W$ with $\tau\geq T$ in case (b). By (\ref{2.vanishintegral}) the $L^2$ norms satisfy
\begin{equation}
\label{A.1}
\|\del\xi_n\|^2 + t^2 \|R\xi_n\|^2\ =\ \|L_{t_n}\xi_n\|^2\ =\ \la_n \to 0
 \end{equation}
 as $n\to\infty$.  On any compact set $K\subset \C\setminus\{\mbox{nodes of $C_0$}\}$ we can use the coordinates of Section~\ref{section5} to identify $K\cap C_s$ with $K\cap C_0$ and regard $\xi_n$ as a section on $K\cap C_0$.  Under this identification,  the geometry of $K\cap C_s$ converges to that of $K\cap C_0$.    An elliptic estimate for $\del$ then  provides  a bound on  the Sobolev $W^{1,2}$ norm of $\xi_n$:
 $$\int_{C_n} |\nabla \xi_n|^2 + |\xi_n|^2\ \leq \
c_1  \int_{C_n} |\del\xi_n|^2 + |\xi_n|^2
\ \leq \  c_2\, (\la_n +1) \ \leq\ 2c_2
$$
for large $n$.  Therefore, by elliptic theory,   a subsequence  converges in $L^2(K)$ and weakly in  $W^{1,2}(K)$  to a limit $\xi_0$  with $L_\tau^*L_\tau\xi_0=0$.   Applying this argument for a sequence of compact sets $K$ that exhaust  $\C\setminus\{\mbox{nodes}\}$  and repeatedly extracting subsequences yields a solution of   $L_\tau\xi_0=0$ on $C_0\setminus\{\mbox{nodes}\}$.  By a standard argument (see the proof of Lemma~7.6 in \cite{LP2}) $\xi_0$ extends over the nodes in the normalization of $C_0$  to a solution of $L_\tau\xi_0=0$.  Theorem~\ref{2.vanishingtheorem} then implies   that $\xi_0$ is holomorphic.

To show $\xi_0$ is non-trivial we must rule out the   possibility that  the $L^2$ norm of $\xi_n$ accumulates at the nodes.  Fix a node $p$ of $C_0$, a local holomorphic section $\nu$ of $\N$ with $\frac12\leq|\nu|^2\leq 2$ pointwise on $C_n(2\ep)=B(p,2\ep)\cap C_n$, and coordinates   $(x,y)$ around $p$ in which $C_n =\{xy=\zeta s_n\}$.  Then the functions $f_n$ defined by  $\xi_n=f_n\nu$ satisfy $|\xi_n|^2\leq 2|f_n|^2$ and $|\del f_n|^2\leq 2 |\del\xi_n|^2$ on $C_n$.  Lemma~\ref{LemmaA2} below and (\ref{A.1}) show that 
 $$
 \int_{C_n(\ep)} |\xi_n|^2
\ \leq\ c_4\ep^2\int_{C_n} |\del \xi_n|^2 \ +\ c_5  \int_{C_n(2\ep)\setminus C_n(\ep)} \hspace{-7mm}| \xi_n|^2
\ \leq\ c_4\ep^2\la_n  + c_5  \int_K | \xi_n|^2
$$
with $ \la_n\to 0$.  If $\xi_0=0$ then the last integral also vanishes as $n\to\infty$ because $\xi_n\to \xi_0=0$ in $L^2(K)$.  Thus the $L^2$ norm does not accumulate at any node, which implies that  $\|\xi_0\| = \lim_{n\to\infty} \|\xi_n\|=1$; this is a contradiction unless $\xi_0\not= 0$.

Furthermore,  $\xi_0$ is continuous, as follows.    Fix a node $p$, a local holomorphic trivialization of $\N\to\C$ around $p$,  and local coordinates  in which $C_s$ is given by $xy=\zeta s$  and regard $\xi_0$ as a holomorphic function.  Let $p'$ and $p''$ be the points in the normalization above $p$ and let $A_n$ be the annular region on $C_n$ between
the circles $\gamma_1(s) =\{x=1\}$ and $\gamma_2(s)=\{y=1\}$.  Setting $\eta=x^{-1}dx=-y^{-1}dy$  we have
$$
2\pi i\,\xi_0(p')\ =\ \int_{\gamma_1(0)}  \xi\,\eta\ =\  \lim_{n\to\infty} \int_{\gamma_1(s_n)}  \xi_n \eta.
$$
and similarly for $\xi_0(p'')$.   Setting $r=|x|$ and noting that  $|\eta|_g^2\, dv_g$ is conformally invariant (cf. Lemma~\ref{LemmaA2}), we have
$$
2\pi \big|\xi_0(p') -\xi_0(p'')\big|
\ \leq \  \varlimsup \int_{A_n} |\del\xi_n|\ |\eta|
\ \leq\ \varlimsup\   \|\del\xi_n\|\left(2\pi \int_{s_n}^1 \frac{r\ dr}{r^2}\right)^{\frac12}
\ \leq\ \varlimsup \big(2\pi  \la_n |\log |s_n||\big)^{\frac12}
\ =\ 0.
$$
Thus $\xi_0$ is a continuous element of $\ker L_{\tau}$ on $C_0$.   Lemma~\ref{WjLemma} then implies that $\xi_0\in \E_0$ in case (a) and $\xi_0\in W$ in case (b).

But in case (a) each $\xi_n$ is $L^2$ orthogonal to $\E_{s_n}$ on $C_n$.  For  the basis $\{\psi_{k,s}\}$  in  (\ref{6.PsiHat}), one sees that
for each $\delta>0$ there is a compact set $K$ so that the $L^2$ norm of $\psi_{k,s}$ on $C_n\setminus K$ is less than $\delta$, uniformly in $s$.   A simple estimate then shows  that $\xi_0$ is $L^2$ orthogonal to $\E_0$.  Likewise, in case (b) one sees that  $\xi_0$ is $L^2$ orthogonal to $W$.  This contradicts our previous conclusion about $\xi_0$,    completing the proof.
\end{proof}

\begin{lemma}
\label{LemmaA2}
Let $C_s(2\ep)$ be the curve $\{xy=\zeta s\,|\, |x|<2\ep, |y|<2\ep\}$ in $\cx^2$ with the induced Riemannian metric.  Then there are constants $c_1$ and $c_2$, independent of $s$ and $\ep$,  such that every smooth function $f$ on $C_s$ satisfies
$$
\int_{C_s(\ep)} |f|^2\ \leq\ c_1\ep^2\int_{C_s(2\ep)} |\del f|^2\ +\ c_2 \int_{C_s(2\ep)\setminus C_s(\ep)} |f|^2.
$$
\end{lemma}

\begin{proof}
A simple calculation shows that the  Riemannian metric $g_s$ on $C_s$ is conformal to the euclidean metric in the $x$-coordinate:
\begin{equation}
\label{6.metric}
g_s=\gamma^2 dx^2 \quad\mbox{where}\quad  \gamma^2=1+\frac{s^2}{r^4}, \quad r=|x|.
\end{equation}
Fix a smooth cutoff function $\beta(\rho)$, $\rho^2=|x|^2+|y|^2$,  supported on $B=B(2\ep)\subset\cx^2$ with $\beta=1$ on $B(\ep)$,  $0\leq \beta\leq 1$  and $|d\beta|\leq 2/\ep$ pointwise.  
Then $h=\beta f$ is a smooth function of $x$ that vanishes on $\partial B$.  Setting $\phi=\frac12(r^2-s^2/r^2)$, we have $dvol_s=\phi'\, dr d\theta$ by (\ref{6.metric}) and can integrate by parts:
$$
I\ =\ \int_B |h|^2\ dvol_s\ =\ \int_B |h|^2 \phi'\ dr d\theta\ \leq\ \int_B |h|\, |dh|\, 2\phi\ dr d\theta.
$$
But $2\phi\leq r^2\gamma^2=\rho^2$ with $\rho\leq 2\ep$ so, continuing using Cauchy-Schwarz and $dvol_s=\gamma^2\, rdrd\theta$, 
$$
I\ \leq\ \int_B |h|\gamma\rho\sqrt{r}\cdot |dh|\sqrt{r}\ dr d\theta
\ \leq\ 2\ep \sqrt{I}\ \left( \int_B |dh|^2\ r\ dr d\theta\right)^{1/2}.
$$
The last integrand is conformally invariant, so can be replaced by $|d h|^2_g\, dv_g$.  Rearranging, we have $I\leq 4\ep^2 \|dh\|^2 \leq 8\ep^2\|\del h\|^2$ where this second inequality is obtained by  integrating  by parts using the formula $2\del^*\del=d^*d$.  The lemma follows because $|\del h|^2\leq 2( |\del\beta|^2 |f|^2+ |\del f|^2)$ where $d\beta$ has support on $C_s(2\ep)\setminus C_s(\ep)$.
\end{proof}

\vspace{3mm}

{\small

\vspace{1cm}

\noindent {\em  Department of  Mathematics,  University of Central Florida, Orlando, FL 32816\\[1mm]
Department of  Mathematics,  Michigan State University, East Lansing, MI 48824}\\[1mm]
{\em e-mail:}\ \ {\ttfamily junlee@mail.ucf.edu \ \ \ {\ttfamily parker\@@math.msu.edu}

\medskip

\noindent\date{\it \today}
}

\end{document}